\documentclass{article}

\usepackage{lineno,hyperref}
\usepackage{amsmath}
\usepackage{amssymb}
\usepackage{amsthm}
\modulolinenumbers[5]
\usepackage{enumitem}
\usepackage{tkz-euclide}

\newtheorem{lemma}{Lemma}
\newtheorem*{lemma*}{Lemma}

\newtheorem{theorem}{Theorem}
\newtheorem*{theorem*}{Theorem}

\theoremstyle{definition}
\newtheorem{example}{Example}

\newcommand{\dotminus}{\mathbin{\dot{-}}}

\newcommand{\dotoplus}{\mathbin{\dot{\oplus}}}

\newcommand{\circumcircle}[4]{
    \tkzDefCircle[circum](#1,#2,#3)
    \tkzGetPoint{#4}
    \tkzDrawCircle(#4,#1)
}

\DeclareMathOperator\mat{M}
\DeclareMathOperator\unit{U}
\DeclareMathOperator\eunit{EU}
\DeclareMathOperator\stunit{StU}

\DeclareMathOperator\glin{GL}

\DeclareMathOperator\stlin{St}

\DeclareMathOperator\ofalin{AL}
\DeclareMathOperator\ofasymp{ASp}
\DeclareMathOperator\ofaorth{AO}

\DeclareMathOperator\Ker{Ker}

\newcommand{\eps}{\varepsilon}
\newcommand{\leqt}{\trianglelefteq}
\newcommand{\inv}[1]{
    \!\;\overline{
        \!\!\:#1\vphantom !\!\!\:
    }\;\!
}

\DeclareMathOperator{\Aut}{Aut}

\newcommand{\up}[2]{{^{#1}\!{#2}}}

\title{
    Groups with \(\mathsf{BC}_\ell\)-commutator relations
}

\author{
      Egor Voronetsky\thanks{
          Research is supported by the Russian Science Foundation grant 22-21-00257.
      } \\
      Chebyshev Laboratory, \\
      St. Petersburg State University, \\
      14th Line V.O., 29B, \\
      Saint Petersburg 199178 Russia \\
}

\begin{document}
\maketitle

\begin{abstract}
Isotropic odd unitary groups generalize Chevalley groups of classical types over commutative rings and their twisted forms. Such groups have root subgroups parameterized by a root system
\(
    \mathsf{BC}_\ell
\) and may be constructed by so-called odd form rings with Peirce decompositions. We show the converse: if a group \(G\) has root subgroups indexed by roots of
\(
    \mathsf{BC}_\ell
\) and satisfying natural conditions, then there is a homomorphism
\(
    \stunit(R, \Delta) \to G
\) inducing isomorphisms on the root subgroups, where
\(
    \stunit(R, \Delta)
\) is the odd unitary Steinberg group constructed by an odd form ring \((R, \Delta)\) with a Peirce decomposition. For groups with root subgroups indexed by
\(
    \mathsf A_\ell
\) (the already known case) the resulting odd form ring is essentially a generalized matrix ring.
\end{abstract}


\section{Introduction}

Let \(K\) be a commutative unital ring and \(\Phi\) be a root system. By Chevalley group \(G(\Phi, K)\) we mean the point group of the simply connected split reductive group scheme over \(K\) of type \(\Phi\). It has so-called root subgroups
\(
    G_\alpha \leq G(\Phi, K)
\) for all
\(
    \alpha \in K
\) isomorphic to the additive group of \(K\) such that Chevalley commutator formula
\[
    [G_\alpha, G_\beta]
    \leq \prod_{ \substack{
        i \alpha + j \beta \in \Phi \\
        i, j \in \{ 1, 2, \ldots \}
    } } G_{i \alpha + j \beta}
\]
holds for non-opposite roots \(\alpha\) and \(\beta\). Moreover, there are Weyl elements from
\(
    G_\alpha G_{-\alpha} G_\alpha
\) permuting these root subgroups.

Conversely, suppose that \(G\) is arbitrary group with distinguished subgroups \(G_\alpha\) indexed by roots \(\alpha\) from a root system satisfying some natural conditions (in particular, they satisfy the Chevalley commutator formula and generate \(G\)). If also for every
\(
    1 \neq x \in G_\alpha
\) there is
\(
    1 \neq y \in G_{-\alpha}
\) such that \(xyx\) is a Weyl element, then \(G\) is almost a Chevalley group over a field \cite{root-group-datum, abs-root-subgr}.

If we only assume that there are some Weyl elements (as for groups over arbitrary commutative rings), then in the simply laced case \(G\) is a factor-group of Steinberg group
\(
    \stlin(\Phi, K)
\) over a ring \(K\) (possibly non-commutative for groups of type
\(
    \mathsf A_\ell
\)) \cite{graded-groups}.

It turns out that such structure result holds over more general class of groups without Weyl elements. For example, if \(R\) is an associative unital ring with a complete family of orthogonal idempotents \(e_1\), \ldots, \(e_n\), then \(\glin(R)\) has root subgroups isomorphic to \(e_i R e_j\) for \(i \neq j\) satisfying Chevalley commutator formula. In \cite{a-com-rel} we proved that if \(G\) is an abstract group with root subgroups
\(
    G_\alpha \leq G
\) indexed by roots of the root system of type
\(
    \mathsf A_\ell
\) for
\(
    \ell \geq 3
\) satisfying some conditions, then there is a homomorphism from the Steinberg group
\(
    \stlin(R)
\) to \(G\) inducing isomorphisms on the root subgroups, where \(R\) is an associative but possibly non-unital ring with a Peirce decomposition of rank \(\ell + 1\).

There exist several constructions \cite{bak, calmes-fasel, knus, inv-book, odd-def-petrov, def-tits, def-weil} generalizing Chevalley groups of classical types
\(
    \mathsf A_\ell
\),
\(
    \mathsf B_\ell
\),
\(
    \mathsf C_\ell
\), and
\(
    \mathsf D_\ell
\). Unitary groups of A. Bak \cite{bak} constructed by modules over form rings with hermitian and quadratic forms generalize Chevalley groups of types
\(
    \mathsf A_\ell
\),
\(
    \mathsf C_\ell
\),
\(
    \mathsf D_\ell
\), and some of their twisted forms. A more general construction of V. Petrov \cite{odd-def-petrov} uses odd form parameters on modules to get all classical Chevalley groups. Twisted forms of classical Chevalley groups also may be constructed using Azumaya algebras with involutions \cite{inv-book}, but this construction does not gives e.g. matrix linear groups over non-commutative rings.

In \cite{thesis, classic-ofa} we further generalized V. Petrov's construction using odd form rings. These objects form a two-sort variety in the sense of universal algebra (unlike odd form parameters on modules). As in other constructions, the unitary groups constructed by odd form rings with additional structure (a Peirce decomposition \cite{thesis}) have natural root subgroups indexed by a root system of type
\(
    \mathsf{BC}_\ell
\). In the classical cases the redundant root subgroups are trivial.

In this paper we prove
\begin{theorem*}
    Let \(G\) be a group with root subgroups indexed by a root system of type
    \(
        \mathsf{BC}_\ell
    \) for
    \(
        \ell \geq 4
    \) and satisfying natural conditions. Then \(G\) is a factor-group of the unitary Steinberg group
    \(
        \stunit(R, \Delta)
    \) constructed by an odd form ring with a Peirce decomposition. Moreover, the root subgroups of the Steinberg group map onto such subgroups of \(G\) with kernels central in
    \(
        \stunit(R, \Delta)
    \).
\end{theorem*}

Conversely, if \((R, \Delta)\) is a sufficiently good odd form ring (in particular, a locally finite odd form \(K\)-algebra) and
\(
    N \leqt \stunit(R, \Delta)
\) trivially intersects the root subgroups, then \(N\) is central by the sandwich classification of normal subgroups in
\(
    \eunit(R, \Delta)
\) (see e.g. \cite{preusser-odd}) and the centrality of \(\mathrm K_2\)-functor (\cite{central-ku2} or \cite[theorem 7]{thesis}). The sandwich classification theorem is known only for unital \((R, \Delta)\) and there are no such \(N\) non-trivially intersecting the root subgroups such that the intersections are central in
\(
    \stunit(R, \Delta)
\).

Informally, the natural conditions from the theorem are the following: \(G\) is generated by the root subgroups, the Chevalley commutator formula holds, and each root subgroup may be expressed in terms of commutators of other root subgroups. If some commutative unital ring \(K\) acts on the root subgroups in a natural way, then the resulting odd form ring \((R, \Delta)\) is an augmented odd form \(K\)-algebra and in the last claim of the theorem the maps of the root subgroups are isomorphisms.

The Chevalley commutator formula is
\[
    [G_\alpha, G_\beta]
    \leq \prod_{ \substack{
        i \alpha + j \beta \in \Phi \\
        i, j \in \{ 1, 2, \ldots \}
    } } G_{i \alpha + j \beta},
\]
we also assume that
\(
    G_{2 \alpha} \leq G_\alpha
\) are \(2\)-step nilpotent filtrations for any ultrashort root \(\alpha\). In other words, \(G\) is a group with
\(
    \mathsf{BC}_\ell
\)-commutator relations in the sense of \cite{st-jordan}. Alternatively, we may assume only that \(G\) contains groups indexed by a root system of type
\(
    \mathsf B_\ell
\) and
\[
    [G_\alpha, G_\beta]
    \leq \prod_{ \substack{
        i \alpha + j \beta \in \Phi \\
        i, j \in \mathbb R_+
    } } G_{i \alpha + j \beta}.
\]
These formulas turn out to be equivalent (modulo other natural conditions) up to a choice of the nilpotent filtrations
\(
    G_{2 \alpha} \leq G_\alpha
\).

The paper is organized as follows. In section \(2\) we recall the definitions of odd form rings and associated unitary groups. In sections \(3\) and \(4\) we list the precise conditions on \(G\) and its subgroups. Namely, the conditions are \ref{comm-rel}--\ref{short-gen} without using a commutative unital ring \(K\) or \ref{comm-rel}--\ref{homo-comm} involving \(K\). In section \(4\) we also discuss the case \(\ell = 3\): under additional ``associativity conditions'' \ref{ass-law-l}--\ref{ass-law-bal} the main results still hold, otherwise there are counterexamples (e.g. Chevalley groups of types
\(
    \mathsf E_6
\) and
\(
    \mathsf E_7
\)). These associativity conditions always hold for
\(
    \ell \geq 4
\) by theorem \ref{ass-forced}. Sections \(5\) and \(6\) contain the proof of the main theorem \ref{coord-cryst} for groups satisfying \ref{comm-rel}--\ref{homo-comm}. In the last section \(7\) we prove theorem \ref{coord-sphere} for groups satisfying only \ref{comm-rel}--\ref{short-gen}.

\section{Odd form groups and odd form rings}

In this paper we build a lot of \(2\)-step nilpotent groups with various operations, so it is useful to develop some technique to simplify such constructions. The group operation of \(2\)-step nilpotent groups is usually denoted by \(\dotplus\). All lemmas in this section may be checked directly or using the machinery of polyquadratic maps \cite[\S 1.3]{thesis}.

We say that \((M, H)\) is a \textit{hermitian group} if \(M\) and \(H\) are abelian groups, there is an automorphism
\(
    \inv{(-)} \colon H \to H
\) of order at most \(2\), and there is a biadditive pairing
\(
    \langle {-}, {=} \rangle
    \colon M \times M
    \to H
\) such that \(
    \inv{\langle m, m' \rangle}
    = \langle m', m \rangle
\) for all \(m, m' \in M\).

Let \((M_1, H_1)\) and \((M_2, H_2)\) be hermitian groups, \(N\) be an abelian group. A \textit{sesquiquadratic map}
\(
    (\alpha, \beta)
    \colon (M_1, H_1) \times N
    \to (M_2, H_2)
\) consists of biadditive
\(
    \alpha \colon M_1 \times N \to M_2
\) and triadditive
\(
    \beta \colon N \times H_1 \times N \to H_2
\) such that
\begin{align*}
    \inv{\beta(n, h, n')}
    &= \beta(n', \inv h, n),
    \\
    \langle \alpha(m, n), \alpha(m', n') \rangle
    &= \beta(n, \langle m, m' \rangle, n').
\end{align*}
Sesquiquadratic maps
\(
    (\alpha_{ij}, \beta_{ij})
    \colon (M_i, H_i) \times N_{ij}
    \to (M_j, H_j)
\) for
\(
    1 \leq i < j \leq 3
\) together with biadditive
\(
    \mu \colon N_{12} \times N_{23} \to N_{13}
\) satisfy the \textit{associative law} if
\begin{align*}
    \alpha_{23}(\alpha_{12}(m, n_{12}), n_{23})
    &= \alpha_{13}(m, \mu(n_{12}, n_{23})),
    \\
    \beta_{23}(
        n_{23},
        \beta_{12}(n_{12}, h, n_{12}'),
        n_{23}'
    )
    &= \beta_{13}(
        \mu(n_{12}, n_{23}),
        h,
        \mu(n_{12}', n_{23}')
    ).
\end{align*}

A \textit{odd form group}
\(
    (M, H, \Delta)
\) consists of a hermitian group \((M, H)\), a group \(\Delta\) with the group operation \(\dotplus\), and maps
\(
    \phi \colon H \to \Delta
\),
\(
    \pi \colon \Delta \to M
\),
\(
    \rho \colon \Delta \to H
\) such that
\begin{align*}
    \phi(h + \inv h) &= \dot 0;
    &
    \phi(h + h') &= \phi(h) \dotplus \phi(h');
    &
    u \dotplus v
    &=
    v \dotplus u
    \dotminus \phi(\langle \pi(u), \pi(v) \rangle);
    \\
    \phi(\langle m, m \rangle) &= \dot 0;
    &
    \pi(u \dotplus v) &= \pi(u) + \pi(v);
    &
    \rho(u \dotplus v)
    &= \rho(u)
    - \langle \pi(u), \pi(v) \rangle
    + \rho(v);
    \\
    \pi(\phi(h)) &= 0;
    &
    \rho(\phi(h)) &= h - \inv h;
    &
    0
    &= \rho(u)
    + \langle \pi(u), \pi(u) \rangle
    + \inv{\rho(u)}.
\end{align*}

The group \(\Delta\) is called an \textit{odd form parameter} (see also \cite[\S2]{odd-def-petrov}), it is \(2\)-step nilpotent with the nilpotent filtration
\(
    \phi(H) \leq \Delta
\). Also,
\(
    \rho(\dot 0) = 0
\) and
\(
    \rho(\dotminus u) = \inv{\rho(u)}
\). The condition
\(
    \phi(\langle m, m \rangle) = \dot 0
\) is unnecessary if
\(
    \pi(\Delta) = M
\), this is usually the case in this paper. In the next lemma we construct free odd form parameters.

\begin{lemma} \label{ofg-free}
    Let \((M, H)\) be a hermitian group with elements \(m_i \in M\), \(h_i \in H\) for \(i \in I\) such that
    \(
        h_i + \langle m_i, m_i \rangle + \inv{h_i} = 0
    \). Then there is the universal odd form group
    \(
        (M, H, \Delta)
    \) with elements
    \(
        u_i \in \Delta
    \) such that
    \(
        \pi(u_i) = m_i
    \) and
    \(
        \rho(u_i) = h_i
    \). Namely,
    \(
        \Delta
        = \phi(H)
        \dotoplus \bigoplus^\cdot_{i \in I}
            \mathbb Z^\cdot u_i
    \), it is the set of formal expressions
    \(
        \phi(h)
        \dotplus \sum_{i \in I}^\cdot
            k_i^\cdot u_i
    \), where almost all
    \(
        k_i \in \mathbb Z
    \) are zero. The kernel of \(\phi\) is generated by the elements
    \(
        h + \inv h
    \) and
    \(
        \langle m, m \rangle
    \).
\end{lemma}

We say that elements
\(
    u_i \in \Delta
\), \(i \in I\) generated the odd form parameter of an odd form group
\(
    (M, H, \Delta)
\) if they generate the abelian group
\(
    \Delta / \phi(H)
\).

\begin{lemma} \label{ofg-fact}
    Let
    \(
        (M, H, \Delta)
    \) be an odd form group and
    \(
        X \leq \Delta
    \) be the subgroup generated by
    \(
        u_i \dotminus v_i
    \) for \(i \in I\). Suppose that
    \(
        \pi(u_i) = \pi(v_i)
    \) and
    \(
        \rho(u_i) = \rho(v_i)
    \). Then
    \(
        X \leqt \Delta
    \) and
    \(
        (M, H, \Delta / X)
    \) is a well-defined odd form group.
\end{lemma}

Let
\(
    (M_1, H_1, \Delta_1)
\) and
\(
    (M_2, H_2, \Delta_2)
\) be odd form groups, \(N\) be an abelian group. A \textit{sesquiquadratic map}
\[
    (\alpha, \beta, \gamma)
    \colon (M_1, H_1, \Delta_1) \times N
    \to (M_2, H_2, \Delta_2)
\]
consists of a sesquiquadratic map
\(
    (\alpha, \beta)
    \colon (M_1, H_1) \times N
    \to (M_2, H_2)
\) and a map
\(
    \gamma \colon \Delta_1 \times N \to \Delta_2
\) such that
\begin{align*}
    \pi(\gamma(u, n)) &= \alpha(\pi(u), n);
    &
    \gamma(\phi(h), n) &= \phi(\beta(n, h, n));
    \\
    \gamma(u \dotplus u', n)
    &= \gamma(u, n) \dotplus \gamma(u', n);
    &
    \rho(\gamma(u, n)) &= \beta(n, \rho(u), n);
    \\
    \gamma(u, n + n')
    &= \gamma(u, n)
    \dotplus \phi(\beta(n', \rho(u), n))
    \dotplus \gamma(u, n').\mkern-200mu
\end{align*}
Sesquiquadratic maps
\(
    (\alpha_{ij}, \beta_{ij}, \gamma_{ij})
    \colon (M_i, H_i, \Delta_i) \times N_{ij}
    \to (M_j, H_j, \Delta_j)
\) for
\(
    1 \leq i < j \leq 3
\) together with biadditive
\(
    \mu \colon N_{12} \times N_{23} \to N_{13}
\) satisfy the \textit{associative law} if
\(
    (\alpha_{ij}, \beta_{ij})
\) together with \(\mu\) satisfy the associative law (in the sense of hermitian groups) and
\[
    \gamma_{23}(\gamma_{12}(u, n_{12}), n_{23})
    = \gamma_{13}(u, \mu(n_{12}, n_{23})).
\]

\begin{lemma} \label{ofg-sq-free}
    Let
    \(
        (M_1, H_1, \Delta_1)
    \) and
    \(
        (M_2, H_2, \Delta_2)
    \) be odd form groups such that the odd form parameter \(\Delta_1\) is free with generators \(u_i\), \(i \in I\). Let also
    \(
        N = \bigoplus_{j \in J} \mathbb Z n_j
    \) be a free abelian group and
    \(
        (\alpha, \beta)
        \colon (M_1, H_1) \times N
        \to (M_2, H_2)
    \) be a sesquiquadratic map. Then for every choice of elements
    \(
        v_{ij} \in \Delta_2
    \) such that
    \(
        \pi(v_{ij}) = \alpha(\pi(u_i), n_j)
    \) and
    \(
        \rho(v_{ij}) = \beta(n_j, \rho(u_i), n_j)
    \) there is unique map
    \(
        \gamma \colon \Delta_1 \times N \to \Delta_2
    \) such that
    \(
        (\alpha, \beta, \gamma)
    \) is sesquiquadratic and
    \(
        \gamma(u_i, n_j) = v_{ij}
    \).
\end{lemma}

\begin{lemma} \label{ofg-sq-fact}
    Let
    \(
        (\alpha, \beta, \gamma)
        \colon (M_1, H_1, \Delta_1) \times N
        \to (M_2, H_2, \Delta_2)
    \) be a sesquiquadratic map,
    \(
        X \leq \Delta_1
    \) be a subgroup generated by
    \(
        u_i \dotminus v_i
    \), \(i \in I\) such that
    \(
        (M_1, H_1, \Delta_1 / X)
    \) is a well-defined odd form group, \(Y \leq N\) be a subgroup generated by \(
        n_j - n'_j
    \), \(j \in J\). If
    \(
        \gamma(u_i, n) = \gamma(v_i, n)
    \) for all generators \(n \in N\),
    \(
        \gamma(u, n_j) = \gamma(u, n'_j)
    \) for all generators
    \(
        u \in \Delta_1
    \), and
    \(
        (\alpha, \beta)
    \) factors though
    \(
        (M_1, H_1) \times N / Y
    \), then
    \(
        (\alpha, \beta, \gamma)
    \) factors through
    \(
        (M_1, H_1, \Delta_1 / X) \times N / Y
    \).
\end{lemma}

\begin{lemma} \label{ofg-sq-ass}
    Let
    \(
        (\alpha_{ij}, \beta_{ij}, \gamma_{ij})
        \colon (M_i, H_i, \Delta_i) \times N_{ij}
        \to (M_j, H_j, \Delta_j)
    \) be sesquiquadratic maps for
    \(
        1 \leq i < j \leq 3
    \) and
    \(
        \mu \colon N_{12} \times N_{23} \to N_{13}
    \) be a biadditive map. Suppose that
    \(
        (\alpha_{ij}, \beta_{ij})
    \) together with \(\mu\) satisfy the associative law and
    \(
        \gamma_{23}(\gamma_{12}(u_1, n_{12}), n_{23})
        = \gamma_{13}(u_1, \mu(n_{12}, n_{23}))
    \) for all generators
    \(
        u \in \Delta_1
    \),
    \(
        n_{12} \in N_{12}
    \),
    \(
        n_{23} \in N_{23}
    \). Then
    \(
        (\alpha_{ij}, \beta_{ij}, \gamma_{ij})
    \) together with \(\mu\) satisfy the associative law.
\end{lemma}

Now we recall the definitions from \cite[\S\S 1.4--1.7, 2.1, 2.2]{thesis} or \cite[\S\S 3, 4]{classic-ofa}. Let \(K\) be a commutative unital ring. If \(R\) is an associative \(K\)-algebra (not necessarily unital), then the semi-direct product
\(
    R \rtimes K
\) is the abelian group
\(
    R \oplus K
\) with the multiplication
\[
    (a \oplus k) (a' \oplus k')
    = (aa' + ka' + ak') \oplus kk',
\]
it is an associative ring with the identity
\(
    0 \oplus 1
\) and ideal \(R\). An involution on \(R\) is an anti-automorphism
\(
    \inv{(-)} \colon R \to R
\) of order at most \(2\), i.e. a \(K\)-linear map such that
\(
    \inv{ab} = \inv{\,b\,} \inv a
\) and
\(
    \inv{\inv a} = a
\). Any such involution continues to
\(
    R \rtimes K
\) by
\(
    \inv{ (a \oplus k) } = \inv a \oplus k
\) and makes \((R, R)\) a hermitian group with
\(
    \langle a, b \rangle = \inv ab
\).

A \textit{odd form \(K\)-algebra}
\(
    (R, \Delta)
\) consists of an associative \(K\)-algebra \(R\) with an involution and an odd form parameter on \((R, R)\) together with a sesquiquadratic map
\[
    (
        (a, b) \mapsto ab,
        (a, b, c) \mapsto \inv a b c,
        (u, b) \mapsto u \cdot b
    )
    \colon (R, R, \Delta) \times (R \rtimes K)
    \to (R, R, \Delta)
\]
satisfying the identity
\(
    u \cdot 1 = u
\) and the associative law. Its \textit{unitary group} is
\[
    \unit(R, \Delta) = \{
        g \in \Delta
    \mid
        \pi(g) = \inv{\rho(g)},
        \pi(g) \inv{\pi(g)} = \inv{\pi(g)} \pi(g)
    \}
\]
with the group operation
\(
    gh = g \cdot \pi(h) \dotplus h \dotplus g
\).

An \textit{augmented odd form \(K\)-algebra}
\(
    (R, \Delta, \mathcal D)
\)  consists of an odd form \(K\)-algebra
\(
    (R, \Delta)
\) and an intermediate subgroup
\(
    \phi(R) \leq \mathcal D \leq \Ker(\pi)
\) (so-called \textit{augmentation}) with a structure of a left \(K\)-module such that
\begin{align*}
    \mathcal D \cdot (R \rtimes K)
    &\subseteq \mathcal D,
    &
    v \cdot k &= k^2 v,
    &
    \phi(ka) &= k \phi(a),
    \\ &&
    \rho(kv) &= k \rho(v),
    &
    kv \cdot b &= k(v \cdot b).
\end{align*}
for all
\(
    v \in \mathcal D
\), \(k \in K\), \(a \in R\),
\(
    b \in R \rtimes K
\).

For every commutative unital ring \(K\) there are natural augmented odd form \(K\)-algebras with ``classical'' unitary groups \cite[\S 2.3]{thesis}, \cite[\S 5]{classic-ofa}, i.e. isomorphic to various general linear groups, symplectic groups, and orthogonal groups over \(K\). Let \(R\) be an associative \(K\)-algebra with an involution. Then
\(
    R \times R
\) admits a group operation
\[
    (a, b) \dotplus (c, d) = (
        a + c,
        b - \inv a c + d
    )
\]
and an action of the multiplicative monoid
\(
    (R \rtimes K)^\bullet
\) given by
\[
    (a, b) \cdot (c, k) = (
        ac + ak,
        \inv{\,c\,} bc + (\inv{\,c\,}b + bc) k + bk^2
    ),
\]
so
\(
    \Delta = \{
        (a, b) \mid b + \inv aa + \inv{\,b\,} = 0
    \}
\) is an
\(
    (R \rtimes K)^\bullet
\)-invariant subgroup of
\(
    R \times R
\). The pair
\(
    (R, \Delta, \mathcal D)
\) is an augmented odd form \(K\)-algebra with
\(
    \phi(a) = (0, a - \inv a)
\),
\(
    \pi(a, b) = a
\),
\(
    \rho(a, b) = b
\), and
\(
    \mathcal D = \{ (0, b) \mid b + \inv{\,b\,} = 0 \}
\). If
\(
    R = \mat(n, K \times K)
\) with the involution induced by the permutation of factors, then we obtain the \textit{linear odd form algebra}
\(
    \ofalin(n, K)
\). If
\(
    R = \mat(2n, K)
\) with the involution
\(
    a \mapsto J^{-1} a^{\mathrm t} J
\) for the matrix \(J\) of the split symplectic form on \(K^{2n}\), then we obtain the \textit{symplectic odd form algebra}
\(
    \ofasymp(2n, K)
\). In the case
\(
    2 \in K^*
\) the same construction using split symmetric bilinear forms gives \textit{orthogonal odd form algebras}
\(
    \ofaorth(n, K)
\), but in general these odd form algebras are constructed in another way.

A \textit{Peirce decomposition} of rank
\(
    \ell \geq 0
\) of an augmented odd form \(K\)-algebra
\(
    (R, \Delta, \mathcal D)
\) is a decomposition
\[
    R
    = \bigoplus_{-\ell \leq i, j \leq \ell} R_{ij},
    \quad \Delta
    = \bigoplus^\cdot_{ \substack{
        -\ell \leq i, j \leq \ell\\
        i + j > 0
    } } \phi(R_{ij})
    \dotoplus \bigoplus^\cdot_{
        -\ell \leq i, j \leq \ell
    } \Delta^i_j
\]
such that
\begin{itemize}
    \item \(R_{ij}\) are \(K\)-submodules and
    \(
        \Delta^i_j
    \) are
    \(
        K^\bullet
    \)-invariant subgroups;
    \item
    \(
        \mathcal D_i = \mathcal D \cap \Delta^0_i
    \) are \(K\)-submodules (this condition is missed in \cite{thesis}),
    \(
        \mathcal D
        = \bigoplus_{
            -\ell \leq i \leq \ell
        }^\cdot \mathcal D_i
    \);
    \item
    \(
        R_{ij} R_{kl} = 0
    \) for \(j \neq k\),
    \(
        R_{ij} R_{jk} \leq R_{ik}
    \),
    \(
        \inv{R_{ij}} = R_{-j, -i}
    \);
    \item
    \(
        \pi \colon \Delta^i_j \to R_{ij}
    \) are isomorphisms and
    \(
        \rho(\Delta^i_j) = 0
    \) for \(i \neq 0\);
    \item
    \(
        \phi(R_{-i, i}) \leq \mathcal D_i
    \),
    \(
        \pi(\Delta^0_i) \leq R_{0i}
    \),
    \(
        \rho(\Delta^0_i) \subseteq R_{-i, i}
    \);
    \item
    \(
        \Delta^i_j \cdot R_{kl} = \dot 0
    \) for \(j \neq k\),
    \(
        \Delta^i_j \cdot R_{jk} \subseteq \Delta^i_k
    \).
\end{itemize}
The inverse map to the isomorphism
\(
    \pi \colon \Delta^i_j \to R_{ij}
\) for \(i \neq 0\) is denoted by
\(
    a \mapsto q_i \cdot a
\), see \cite[lemma 7]{thesis} for details. For simplicity let \textit{graded odd form \(K\)-algebras} be augmented odd form \(K\)-algebras with chosen Peirce decompositions.

Let
\(
    (R, \Delta, \mathcal D)
\) be a graded odd form \(K\)-algebra. \textit{Elementary transvections} in its unitary group are
\begin{itemize}

    \item
    \(
        T_{ij}(a)
        = q_i \cdot a
        \dotminus q_{-j} \cdot \inv a
        \dotminus \phi(a)
    \) for \(a \in R_{ij}\) and distinct \(0\), \(|i|\), \(|j|\);

    \item
    \(
        T_i(u)
        = u
        \dotplus q_{-i} \cdot (\rho(u) - \inv{\pi(u)})
        \dotminus \phi(\rho(u) + \pi(u))
    \) for
    \(
        u \in \Delta^0_i
    \) and \(i \neq 0\).

\end{itemize}
They satisfy the so-called \textit{Steinberg relations}, where all indices are non-zero and with distinct absolute values:
\begin{align*}
    T_{ij}(a)\, T_{ij}(b) &= T_{ij}(a + b)
    &
    T_{ij}(a) &= T_{-j, -i}(-\inv a)
    \\
    T_j(u)\, T_j(v) &= T_j(u \dotplus v)
    &
    [T_{ij}(a), T_{jk}(b)] &= T_{ik}(ab)
    \\
    [T_{ij}(a), T_{kl}(b)] &= 1
    &
    [T_{-i, j}(a), T_{ji}(b)] &= T_i(\phi(ab))
    \\
    [T_{ij}(a), T_{ik}(b)] &= 1
    &
    [T_i(u), T_j(v)]
    &=
    T_{-i, j}(-\inv{\pi(u)} \pi(v))
    \\
    [T_i(u), T_{jk}(a)] &= 1
    &
    [T_i(u), T_{ij}(a)]
    &=
    T_{-i, j}(\rho(u) a)\,
    T_j(\dotminus u \cdot (-a)).
\end{align*}

We say that a graded odd form \(K\)-algebra satisfies the \textit{idempotency conditions} if
\begin{itemize}
    \item
    \(
        R_{ij} = \langle
            R_{ik} R_{kj},
            R_{i, -k} R_{-k, j}
        \rangle
    \) for all \(i\), \(j\), and \(k \neq 0\);
    \item
    \(
        \Delta^0_i = \langle
            \Delta^0_j \cdot R_{ji},
            \Delta^0_{-j} \cdot R_{-j, i},
            \phi(R_{-i, i})
        \rangle
    \) for all \(i\) and \(j \neq 0\);
    \item
    \(
        \dot 0
        \to \mathcal D_i
        \to \Delta^0_i
        \xrightarrow{\pi} R_{0i}
        \to 0
    \) are short exact sequences for all \(i\),
\end{itemize}
where \(AB\) and \(A \cdot B\) are the corresponding Minkowski operations on sets.

\section{Groups with \(\mathsf{BC}_\ell\)-commutator relations}

Let
\[
    \Phi = \{
        \pm \mathrm e_i \pm \mathrm e_j,
        \pm \mathrm e_i
    \mid
        1 \leq i, j \leq \ell;\,
        i \neq j
    \} \subseteq \mathbb R^\ell
\]
be a root system of type
\(
    \mathsf B_\ell
\) for some
\(
    \ell \geq 3
\). The Weyl group
\(
    \mathrm W(\Phi)
    = \mathbb Z / 2 \mathbb Z \wr \mathrm S_\ell
\) acts on \(\Phi\) by permutations and sign changes of the coordinates. It is convenient to set
\(
    \mathrm e_{-i} = -\mathrm e_i
\) for all
\(
    1 \leq i \leq \ell
\). All integer indices in the paper run from \(-\ell\) to \(\ell\) unless otherwise specified.

We say that a subset
\(
    \Sigma \subseteq \Phi
\) is \textit{special closed} if it is an intersection of \(\Phi\) and a convex cone not containing opposite non-zero vectors. An element
\(
    \alpha \in \Sigma
\) of a special closed subset is called extreme if it is an extreme point of the convex conical hull of \(\Sigma\), i.e. if there is a hyperplane \(H\) such that
\(
    \{ \alpha \} = \Sigma \cap H
\) and
\(
    \Sigma \setminus \{ \alpha \}
\) lies in one of the half-spaces with boundary \(H\). Every non-empty special closed subset \(\Sigma\) contains at least one extreme element \(\alpha\) and
\(
    \Sigma \setminus \alpha
\) is also special closed for such \(\alpha\). We say that a linear order on a special closed subset
\(
    \Sigma \subseteq \Phi
\) is \textit{right extreme} if either
\(
    \Sigma = \varnothing
\) of the largest element
\(
    \alpha_{\mathrm{max}} \in \Sigma
\) is extreme and the induced order on
\(
    \Sigma \setminus \{ \alpha_{\mathrm{max}} \}
\) is also right extreme.

Let \(G\) be a group with a family of subgroups
\(
    G_\alpha \leq G
\) for
\(
    \alpha \in \Phi
\) satisfying
\begin{enumerate}[label = {(C\arabic*)}]

    \item \label{comm-rel}
    \(
        [G_\alpha, G_\beta] \leq \langle
            G_{i \alpha + j \beta}
        \mid
            i \alpha + j \beta \in \Phi;\,
            i, j \in \mathbb R_+
        \rangle
    \) for all
    \(
        \alpha, \beta \in \Phi
    \) such that
    \(
        \alpha \neq \pm \beta
    \).

    \item \label{non-deg}
    \(
        G_\alpha \cap \langle
            G_\beta
        \mid
            \beta \in \Sigma \setminus \{ \alpha \}
        \rangle = 1
    \) for any special closed \(\Sigma\) and extreme
    \(
        \alpha \in \Sigma
    \).

    \item \label{elementary} \(G\) is generated by all \(G_\alpha\).

    \item \label{long-gen}
    \(
        G_{\mathrm e_i + \mathrm e_j}
        \leq \up{\langle
            G_{\mathrm e_i - \mathrm e_j},
            G_{\mathrm e_j - \mathrm e_i},
            G_{\mathrm e_k},
            G_{-\mathrm e_k}
        \rangle}{\langle
            G_{\mathrm e_i + \mathrm e_k},
            G_{\mathrm e_i},
            G_{\mathrm e_i - \mathrm e_k},
            G_{\mathrm e_j + \mathrm e_k},
            G_{\mathrm e_j},
            G_{\mathrm e_j - \mathrm e_k}
        \rangle}
    \) for distinct \(|i|\), \(|j|\), \(|k|\).

    \item \label{short-gen}
    \(
        G_{\mathrm e_i} \leq \up{\langle
            G_{\mathrm e_j},
            G_{-\mathrm e_j}
        \rangle}{\langle
            G_{\mathrm e_i + \mathrm e_j},
            G_{\mathrm e_i - \mathrm e_j}
        \rangle}
    \) for \(|i| \neq |j|\).

\end{enumerate}

The roots appearing in \ref{long-gen} are shown in figure \ref{stereo} after projection onto the unit sphere and stereographic projection from
\(
    -\frac{\sqrt 2}2 (\mathrm e_1 + \mathrm e_2)
\). Lines and circles correspond to root subsystems of rank \(2\), the dashed ones are of type
\(
    \mathsf A_1 + \mathsf A_1
\). The property \ref{long-gen} means that if \(\alpha\) is long and
\(
    \alpha \in \Psi \subseteq \Phi
\) is a root subsystem of type
\(
    \mathsf{BC}_3
\), then \(G_\alpha\) is contained in the smallest subgroup containing \(G_\beta\) for all
\(
    \beta \in \Psi
\) at an acute angle to \(\alpha\) and normalized by \(G_\beta\) for all
\(
    \beta \in \Psi
\) orthogonal to \(\alpha\). Replacing ``long'' by ``short'' and
\(
    \mathsf{BC}_3
\) by
\(
    \mathsf{BC}_2
\) we obtain \ref{short-gen}.

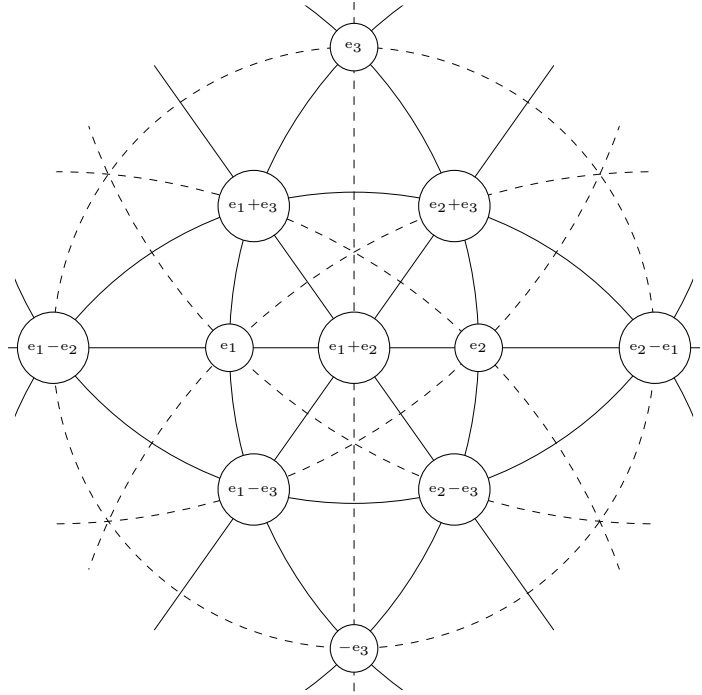
\begin{figure}[h]
    \centering
    \begin{tikzpicture}[scale=4.0]

        \tkzDefPoint(0,0){p1p2}
        \tkzDefPoint(-1,0){p1m2}
        \tkzDefPoint(1,0){m1p2}
        \tkzDefPoint(-1/3,\tkzSqrTwo/3){p1p3}
        \tkzDefPoint(-1/3,-\tkzSqrTwo/3){p1m3}
        \tkzDefPoint(1/3,\tkzSqrTwo/3){p2p3}
        \tkzDefPoint(1/3,-\tkzSqrTwo/3){p2m3}
        \tkzDefPoint(1-\tkzSqrTwo,0){p1}
        \tkzDefPoint(\tkzSqrTwo-1,0){p2}
        \tkzDefPoint(\tkzSqrTwo+1,0){m1}
        \tkzDefPoint(-\tkzSqrTwo-1,0){m2}
        \tkzDefPoint(0,1){p3}
        \tkzDefPoint(0,-1){m3}
        \tkzDefPoint(1.15,0){Clip}

        \tkzClipCircle(p1p2,Clip)

        \tkzDrawLine[color=black,style=dashed,line width=0.2pt,add=2 and 2](p3,m3)
        \tkzDrawLines[color=black,line width=0.2pt,add=2 and 2](p1m2,m1p2 p1p3,p2m3 p2p3,p1m3)

        \tkzSetUpCircle[color=black,line width=0.2pt]
        \circumcircle{p1m2}{p1p3}{p2p3}{Cp}
        \circumcircle{p1m2}{p1m3}{p2m3}{Cm}
        \circumcircle{p3}{p1}{m3}{O1}
        \circumcircle{p3}{p2}{m3}{O2}
        \tkzSetUpCircle[color=black,line width=0.2pt,style=dashed]
        \tkzDrawCircles(p1p2,p3)
        \circumcircle{p1}{p2p3}{m1}{S2p}
        \circumcircle{p2}{p1p3}{m2}{S1p}
        \circumcircle{p1}{p2m3}{m1}{S2m}
        \circumcircle{p2}{p1m3}{m2}{S1m}

        \tkzSetUpPoint[size=27,fill=white]
        \tkzDrawPoints(p1p2,p1m2,m1p2,p1p3,p1m3,p2p3,p2m3)
        \tkzSetUpPoint[size=18,fill=white]
        \tkzDrawPoints(p1,p2,p3,m3)

        \tkzSetUpLabel[centered, font=\tiny]
        \tkzLabelPoint(p1){\(\mathrm e_1\)}
        \tkzLabelPoint(p2){\(\mathrm e_2\)}
        \tkzLabelPoint(p3){\(\mathrm e_3\)}
        \tkzLabelPoint(m3){\(-\mathrm e_3\)}
        \tkzLabelPoint(p1p3){\(
            \mathrm e_1 {+} \mathrm e_3
        \)}
        \tkzLabelPoint(p1p2){\(
            \mathrm e_1 {+} \mathrm e_2
        \)}
        \tkzLabelPoint(p2p3){\(
            \mathrm e_2 {+} \mathrm e_3
        \)}
        \tkzLabelPoint(p2m3){\(
            \mathrm e_2 {-} \mathrm e_3
        \)}
        \tkzLabelPoint(m1p2){\(
            \mathrm e_2 {-} \mathrm e_1
        \)}
        \tkzLabelPoint(p1m2){\(
            \mathrm e_1 {-} \mathrm e_2
        \)}
        \tkzLabelPoint(p1m3){\(
            \mathrm e_1 {-} \mathrm e_3
        \)}

    \end{tikzpicture}

    \caption{Stereographic projection of \(\mathsf B_3\)} \label{stereo}
\end{figure}

For example, if
\(
    (R, \Delta, \mathcal D)
\) is a graded odd form algebra satisfying the idempotency conditions, then the elementary group
\(
    \eunit(R, \Delta)
    = \langle T_{ij}(a), T_i(u) \rangle
    \leq \unit(R, \Delta)
\) satisfies \ref{comm-rel}--\ref{short-gen}. Indeed, \ref{non-deg} follows from \cite[lemma 10]{thesis}. The Steinberg group
\(
    \stunit(R, \Delta)
\) is the abstract group generated by \(X_{ij}(a)\) for
\(
    a \in R_{ij}
\),
\(
    0 \neq |i| \neq |j| \neq 0
\) and \(X_i(u)\) for
\(
    u \in \Delta^0_i
\),
\(
    i \neq 0
\) satisfying the Steinberg relations, so it also satisfies \ref{comm-rel}--\ref{short-gen}. Note that here the idempotency conditions are used only for \(R_{ij}\) and
\(
    \Delta^0_i
\) with non-zero \(i\) and \(j\), these cases are a weakening of the firmness condition \cite[\S 3.3]{thesis}.

\begin{lemma} \label{unip-subgr}
    Let
    \(
        \Sigma \subseteq \Phi
    \) be a special closed subset with a right extreme order. Then the product map from the Cartesian product
    \(
        \prod^{\mathbf{Set}}_{
            \alpha \in \Sigma
        } G_\alpha
    \) to \(G\) is injective and its image
    \(
        \prod_{\alpha \in \Sigma} G_\alpha
    \) is a subgroup.
\end{lemma}
\begin{proof}
    We prove that
    \(
        X = \prod_{\alpha \in \Sigma} G_\alpha
    \) is a subgroup by induction on \(|\Sigma|\). Since \(1 \in X\), it suffice to check that
    \(
        G_\beta X \subseteq X
    \) for all
    \(
        \beta \in \Sigma
    \). Let
    \(
        \mu \in \Sigma
    \) be the largest root and
    \(
        Y = \prod_{
            \alpha \in \Sigma \setminus \{ \mu \}
        } G_\alpha
    \), it is a subgroup by the induction hypothesis. If
    \(
        \beta \neq \mu
    \), then
    \[
        G_\beta X
        = G_\beta Y G_\mu
        \subseteq Y G_\mu
        = X.
    \]
    Otherwise note that \(Y\) is normalized by \(G_\mu\) by \ref{comm-rel}, so again
    \[
        G_\mu X
        = G_\mu Y G_\mu
        \subseteq Y G_\mu
        = X.
    \]
    The first claim follows by easy induction on \(|\Sigma|\) using \ref{non-deg}.
\end{proof}

Now let us construct an analogue of a graded odd form algebra by \(G\) using only \ref{comm-rel} and \ref{non-deg}. Let
\(
    T_{ij}
    \colon R_{ij}
    \to G_{\mathrm e_j - \mathrm e_i}
\) and
\(
    T_i \colon \Delta^0_i \to G_{\mathrm e_i}
\) be arbitrary group isomorphisms for non-zero
\(
    |i| \neq |j|
\) and some groups \(R_{ij}\),
\(
    \Delta^0_i
\). The group operations of \(R_{ij}\) and
\(
    \Delta^0_i
\) are denoted by \(+\) and \(\dotplus\) respectively. Since \(R_{ij}\) and \(R_{-j, -i}\) parameterize the same subgroup of \(G\), there are anti-isomorphisms
\(
    R_{ij} \to R_{-j, -i},\, a \mapsto \inv a
\) such that
\(
    \inv{\inv a} = a
\). The non-trivial commutator relations \ref{comm-rel} may be written as the following Steinberg relations, where all indices are assumed to be non-zero and with distinct absolute values:
\begin{enumerate}[label = {(St\arabic*)}, start = 1]

    \item \label{st-a2-ob}
    \(
        [T_{ij}(a), T_{jk}(b)] = T_{ik}(ab)
    \) for unique
    \(
        ab \in R_{ik}
    \);

    \item \label{st-b2-ll}
    \(
        [T_{-i, j}(a), T_{ji}(b)] = T_i(a * b)
    \) for unique
    \(
        a * b \in \Delta^0_i
    \);

    \item \label{st-b2-ss}
    \(
        [T_i(u), T_j(v)] = T_{-i, j}(-u \circ v)
    \) for unique
    \(
        u \circ v \in R_{-i, j}
    \);

    \item \label{st-b2-ob}
    \(
        [T_i(u), T_{ij}(a)]
        = T_{-i, j}(u \triangleleft a)\,
        T_j(\dotminus u \cdot (-a))
    \) for unique
    \(
        u \triangleleft a \in R_{-i, j}
    \),
    \(
        u \cdot (-a) \in \Delta^0_j
    \).

\end{enumerate}
In the case of unitary Steinberg groups we have
\(
    a * b = \phi(ab)
\),
\(
    u \circ v = \inv{\pi(u)} \pi(v)
\), and
\(
    u \triangleleft a = \rho(u) a
\).

Let us write some identities between operations appearing in \ref{st-a2-ob}--\ref{st-b2-ob}. We obtain them using lemma \ref{unip-subgr} in the following ways:
\begin{itemize}
    \item By applying two different parametrizations of \(G_\alpha\) with long \(\alpha\) the same commutator formula may be written in two different ways, so we obtain some identities involving two variables.
    \item By pairwise permuting the factors in the product
    \(
        G_\alpha G_\alpha G_\beta
    \) (or
    \(
        G_\beta G_\alpha G_\alpha
    \)) for linearly independent roots \(\alpha\) and \(\beta\) using commutator formulae in two ways (either applying it twice or multiplying the factors from \(G_\alpha\) and applying the commutator formula once). In such a way we obtain identities involving three variables with domains corresponding to \(\alpha\), \(\alpha\), and \(\beta\).
    \item By pairwise permuting the factors in the product
    \(
        G_\alpha G_\beta G_\gamma
    \) for pairwise linearly independent roots \(\alpha\), \(\beta\), \(\gamma\) using commutator formulae in two ways until the order on the original factors is reversed (starting from swapping the left two factors or the right ones). In such a way we obtain identities involving three variables with domains corresponding to \(\alpha\), \(\beta\), and \(\gamma\). The non-trivial identities appear only if one of \(\alpha\), \(\beta\), \(\gamma\) is a convex linear combination of other two or all three roots are linearly independent and there are roots in the open simplicial cone spanned by them.
\end{itemize}
The resulting identities we simplify using the previously obtained identities and sometimes replacing variables by their opposites. In the list below all indices are non-zero and with distinct absolute values.

\begin{enumerate}[label = {(R\arabic*)}]

    \item \label{mul-cen}
    \(
        ab + c = c + ab
    \) for
    \(
        a \in R_{ij}
    \),
    \(
        b \in R_{jk}
    \),
    \(
        c \in R_{ik}
    \);

    \item \label{mul-dis-l}
    \(
        (a + b)c = ac + bc
    \) for
    \(
        a, b \in R_{ij}
    \),
    \(
        c \in R_{jk}
    \);

    \item \label{mul-dis-r}
    \(
        a(b + c) = ab + ac
    \) for
    \(
        a \in R_{ij}
    \),
    \(
        b, c \in R_{jk}
    \);

    \item \label{inv-mul}
    \(
        \inv{ab} = \inv{\,b\,} \inv a
    \) for
    \(
        a \in R_{ij}
    \),
    \(
        b \in R_{jk}
    \);

    \item \label{str-cen}
    \(
        a * b \dotplus u = u \dotplus a * b
    \) for
    \(
        a \in R_{-i, j}
    \),
    \(
        b \in R_{ji}
    \),
    \(
        u \in \Delta^0_i
    \);

    \item \label{str-dis-l}
    \(
        (a + b) * c = a * c \dotplus b * c
    \) for
    \(
        a, b \in R_{-i, j}
    \),
    \(
        c \in R_{ji}
    \);

    \item \label{str-dis-r}
    \(
        a * (b + c) = a * b \dotplus a * c
    \) for
    \(
        a \in R_{-i, j}
    \),
    \(
        b, c \in R_{ji}
    \);

    \item \label{str-inv}
    \(
        a * b \dotplus \inv{\,b\,} * \inv a = \dot 0
    \) for
    \(
        a \in R_{-i, j}
    \),
    \(
        b \in R_{ji}
    \);

    \item \label{cir-cen}
    \(
        a + u \circ v = u \circ v + a
    \) for
    \(
        u \in \Delta^0_i
    \),
    \(
        v \in \Delta^0_j
    \),
    \(
        a \in R_{-i, j}
    \);

    \item \label{cir-dis-l}
    \(
        (u \dotplus v) \circ w
        = u \circ w + v \circ w
    \) for
    \(
        u, v \in \Delta^0_i
    \),
    \(
        w \in \Delta^0_j
    \);

    \item \label{cir-dis-r}
    \(
        u \circ (v \dotplus w)
        = u \circ v + u \circ w
    \) for
    \(
        u \in \Delta^0_i
    \),
    \(
        v, w \in \Delta^0_j
    \);

    \item \label{inv-cir}
    \(
        \inv{ u \circ v } = v \circ u
    \) for
    \(
        u \in \Delta^0_i
    \),
    \(
        v \in \Delta^0_j
    \);

    \item \label{cir-str-b1}
    \(
        u \circ (a * b)
        + u \triangleleft (-\inv a)
        + b
        = b
        + u \triangleleft (-\inv a)
    \) for
    \(
        u \in \Delta^0_i
    \),
    \(
        a \in R_{-j, -i}
    \),
    \(
        b \in R_{-i, j}
    \);

    \item \label{dot-cen}
    \(
        u \cdot a
        \dotplus v
        = v
        \dotplus u \cdot a
        \dotminus \inv a * (u \circ v)
    \) for
    \(
        u \in \Delta^0_i
    \),
    \(
        a \in R_{ij}
    \),
    \(
        v \in \Delta^0_j
    \);

    \item \label{dot-dis-l}
    \(
        (u \dotplus v) \cdot a
        = u \cdot a
        \dotplus v \cdot a
    \) for
    \(
        u, v \in \Delta^0_i
    \),
    \(
        a \in R_{ij}
    \);

    \item \label{tri-dis-l}
    \(
        (u \dotplus v) \triangleleft a
        = v \triangleleft a
        + u \circ (v \cdot (-a))
        + u \triangleleft a
    \) for
    \(
        u, v \in \Delta^0_i
    \),
    \(
        a \in R_{ij}
    \);

    \item \label{tri-dis-r}
    \(
        u \triangleleft (a + b)
        = u \triangleleft a
        + u \triangleleft b
    \) for
    \(
        u \in \Delta^0_i
    \),
    \(
        a, b \in R_{ij}
    \);

    \item \label{dot-dis-r}
    \(
        u \cdot (a + b)
        = u \cdot a
        \dotplus \inv{\,b\,} * (u \triangleleft a)
        \dotplus u \cdot b
    \) for
    \(
        u \in \Delta^0_i
    \),
    \(
        a, b \in R_{ij}
    \);

    \item \label{mul-ass}
    \(
        (ab)c = a(bc)
    \) for
    \(
        a \in R_{ij}
    \),
    \(
        b \in R_{jk}
    \),
    \(
        c \in R_{kl}
    \);

    \item \label{str-ass}
    \(
        ab * c = a * bc
    \) for
    \(
        a \in R_{-i, j}
    \),
    \(
        b \in R_{jk}
    \),
    \(
        c \in R_{ki}
    \);

    \item \label{cir-str-b2}
    \(
        u \circ (a * b) = 0
    \) for
    \(
        u \in \Delta^0_i
    \),
    \(
        a \in R_{-j, k}
    \),
    \(
        b \in R_{kj}
    \);

    \item \label{cir-dot}
    \(
        u \circ (v \cdot a) = (u \circ v) a
    \) for
    \(
        u \in \Delta^0_i
    \),
    \(
        v \in \Delta^0_j
    \),
    \(
        a \in R_{jk}
    \);

    \item \label{inv-tri}
    \(
        \inv a (u \triangleleft b)
        + (u \cdot a) \circ (u \cdot b)
        + \inv{ (u \triangleleft a) } b
        = 0
    \) for
    \(
        u \in \Delta^0_i
    \),
    \(
        a \in R_{ij}
    \),
    \(
        b \in R_{ik}
    \);

    \item \label{tri-str}
    \(
        (a * b) \triangleleft c
        = a(bc)
        - \inv{\, b \,} (\inv ac)
    \) for
    \(
        a \in R_{-i, j}
    \),
    \(
        b \in R_{ji}
    \),
    \(
        c \in R_{ik}
    \);

    \item \label{dot-str}
    \(
        (a * b) \cdot c = \inv{\,c\,} a * bc
    \) for
    \(
        a \in R_{-i, j}
    \),
    \(
        b \in R_{ji}
    \),
    \(
        c \in R_{ik}
    \);

    \item \label{tri-ass}
    \(
        (u \triangleleft a) b = u \triangleleft ab
    \) for
    \(
        u \in \Delta^0_i
    \),
    \(
        a \in R_{ij}
    \),
    \(
        b \in R_{jk}
    \);

    \item \label{tri-dot}
    \(
        (\dotminus u \cdot a) \triangleleft b
        = \inv{(u \triangleleft a)} (ab)
    \) for
    \(
        u \in \Delta^0_i
    \),
    \(
        a \in R_{ij}
    \),
    \(
        b \in R_{jk}
    \);

    \item \label{dot-ass}
    \(
        (u \cdot a) \cdot b = u \cdot ab
    \) for
    \(
        u \in \Delta^0_i
    \),
    \(
        a \in R_{ij}
    \),
    \(
        b \in R_{jk}
    \).

\end{enumerate}

We say that a family of abstract groups \(R_{ij}\) for distinct non-zero \(|i|\), \(|j|\) and
\(
    \Delta^0_i
\) for \(i \neq 0\) together with operations
\begin{align*}
    \inv{(-)} &\colon R_{ij} \to R_{-j, -i},
    &
    ({-}) \circ ({=})
    &\colon \Delta^0_i \times \Delta^0_j
    \to R_{-i, j},
    \\
    ({-})({=})
    &\colon R_{ij} \times R_{jk}
    \to R_{ik},
    &
    ({-}) \triangleleft ({=})
    &\colon \Delta^0_i \times R_{ij}
    \to R_{-i, j},
    \\
    ({-}) * ({=})
    &\colon R_{-i, j} \times R_{ji}
    \to \Delta^0_i,
    &
    ({-}) \cdot ({=})
    &\colon \Delta^0_i \times R_{ij}
    \to \Delta^0_j
\end{align*}
satisfying
\(
    \inv{a + b} = \inv{\,b\,} + \inv a
\),
\(
    \inv{\inv a} = a
\), and \ref{mul-cen}--\ref{dot-ass} is a \textit{partial graded odd form ring}. Consider two additional conditions, where the indices are non-zero and with distinct absolute values:
\begin{enumerate}[label = {(R\arabic*)}, start = 29]

    \item \label{long-idem} \(
        R_{ij} = \langle
            R_{ik} R_{kj},
            R_{i, -k} R_{-k, -j}
        \rangle
    \);

    \item \label{short-idem} \(
        \Delta^0_i = \langle
            \Delta^0_j \cdot R_{ji},
            \Delta^0_{-j} \cdot R_{-j, i},
            R_{-i, j} * R_{ji}
        \rangle
    \).

\end{enumerate}

\begin{lemma} \label{idem-par-ofr}
    If \(G\) is a group with subgroups \(G_\alpha\) satisfying all \ref{comm-rel}--\ref{short-gen}, then the associated partial graded odd form ring satisfies \ref{long-idem} and \ref{short-idem} (in particular, \(G\) is perfect). For any partial graded odd form ring these two conditions imply that \(R_{ij}\) is abelian and
    \(
        \Delta^0_i
    \) is \(2\)-step nilpotent with nilpotent filtration \(
        \mathcal D_i^{\mathrm{min}} = \langle
            R_{-i, j} * R_{ji}
        \mid
            j \notin \{ 0, i, -i \}
        \rangle
    \). Also,
    \(
        \mathcal D_i^{\mathrm{min}} \circ \Delta^0_j
        = 0
    \),
    \(
        \mathcal D_i^{\mathrm{min}} \cdot R_{ij}
        \subseteq \mathcal D_j^{\mathrm{min}}
    \), and
    \(
        \mathcal D^{\mathrm{min}}_i = \langle
            R_{-i, j} * R_{ji}
        \rangle
    \) for distinct non-zero \(|i|\), \(|j|\).
\end{lemma}
\begin{proof}
    First of all, \ref{short-gen} implies \ref{short-idem}. The condition \ref{long-gen} means that \(R_{ij}\) is generated by
    \(
        R_{ik} R_{kj}
    \),
    \(
        R_{i, -k} R_{-k, j}
    \),
    \(
        \Delta^0_{-i} \circ \Delta^0_j
    \),
    \(
        \Delta^0_{-i} \triangleleft R_{-i, j}
    \), and
    \(
        \inv{ (\Delta^0_j \triangleleft R_{j, -i}) }
    \) for any
    \(
        k \notin \{ 0, i, -i, j, -j \}
    \). But
    \[
        \Delta^0_{-i} \triangleleft R_{-i, j}
        \subseteq \bigl\langle
            (
                \Delta^0_{\pm k} \cdot R_{\pm k, -i}
            ) \triangleleft R_{-i, j},
            (
                R_{i, k} * R_{k, -i}
            ) \triangleleft R_{-i, j},
            \Delta^0_{-i} \circ \Delta^0_j
        \bigr\rangle
        \subseteq \langle
            R_{i, \pm k} R_{\pm k, j},
            \Delta^0_{-i} \circ \Delta^0_j
        \rangle
    \]
    by \ref{tri-dis-l}, \ref{tri-str}, \ref{tri-dot}, \ref{short-idem}, so these generators may be omitted and similarly for
    \(
        \inv{ (\Delta^0_j \triangleleft R_{j, -i}) }
    \). The generators
    \(
        \Delta^0_{-i} \circ \Delta^0_j
    \) may be omitted by \ref{cir-dis-r}, \ref{cir-str-b2}, \ref{cir-dot}, and \ref{short-idem}.

    Now assume that
    \(
        (R_{ij}, \Delta^0_j)_{ij}
    \) is a partial graded odd form ring satisfying \ref{long-idem} and \ref{short-idem}. Then \ref{mul-cen} implies that \(R_{ij}\) is abelian. The claim
    \(
        \mathcal D^{\mathrm{min}}_i \circ \Delta^0_j
        = 0
    \)
    follows from \ref{cir-str-b1} and \ref{cir-str-b2}. The group
    \(
        \mathcal D^{\mathrm{min}}_i
    \) is central in
    \(
        \Delta^0_i
    \) by \ref{str-cen} and contains its commutator by \ref{str-cen}, \ref{dot-cen}, \ref{short-idem}. We have
    \(
        R_{-i, k} * R_{ki}
        \subseteq \langle R_{-i, j} R_{ji} \rangle
    \) for distinct non-zero \(|i|\), \(|j|\), \(|k|\) by \ref{str-dis-l}, \ref{str-inv}, \ref{str-ass}, \ref{long-idem}, so
    \(
        \mathcal D^{\mathrm{min}}_i
        = \langle R_{-i, j} * R_{ji} \rangle
    \) for any
    \(
        j \notin \{ 0, i, -i \}
    \). Finally,
    \(
        \mathcal D_i^{\mathrm{min}} \cdot R_{ij}
        \subseteq \mathcal D_j^{\mathrm{min}}
    \) by \ref{dot-dis-l} and \ref{dot-str}.
\end{proof}

Any graded odd form algebra without the augmentation and the components \(R_{0i}\), \(R_{i0}\),
\(
    \Delta^0_0
\) is a partial graded odd form ring and the idempotency conditions imply \ref{long-idem} and \ref{short-idem}. In theorem \ref{coord-sphere} we show that every partial graded odd form ring satisfying \ref{long-idem} and \ref{short-idem} for
\(
    \ell \geq 4
\) arises from a group with subgroups satisfying \ref{comm-rel}--\ref{short-gen}.

\section{Action of commutative ring}

Recall from \cite[\S 2.1]{thesis} and \cite[\S 4]{classic-ofa} that a \textit{\(2\)-step nilpotent \(K\)-module} is a pair \((M, M_0)\) such that
\begin{itemize}

    \item \(M\) is a group with the group operation \(\dotplus\) and
    \(
        M_0 \leq M
    \) is its nilpotent filtration (i.e.
    \(
        [M, M] \leq M_0
    \) and \(M_0\) is central);

    \item the multiplicative monoid
    \(
        K^\bullet
    \) acts on \(M\) from the right by group endomorphisms;

    \item \(M_0\) has a structure of a left \(K\)-module;

    \item
    \(
        [m \cdot k, m' \cdot k'] = kk' [m, m']
    \);

    \item
    \(
        m \cdot (k + k')
        = m \cdot k
        \dotplus kk' \tau(m)
        \dotplus m \cdot k'
    \) for some (uniquely determined) element
    \(
        \tau(m) \in M_0
    \);

    \item
    \(
        m \cdot k = k^2 m
    \) for
    \(
        m \in M_0
    \).

\end{itemize}
Clearly,
\(
    (\Delta^0_i, \mathcal D_i)
\) are \(2\)-step nilpotent \(K\)-modules with
\(
    \tau(u) = u \dotplus u \cdot (-1) = \phi(\rho(u))
\) for each graded odd form \(K\)-algebra
\(
    (R, \Delta, \mathcal D)
\).

Suppose that \(G\) is a group with a family of subgroups \(G_\alpha\) for
\(
    \alpha \in \Phi
\) satisfying \ref{comm-rel}--\ref{short-gen}, where \(\Phi\) is a root system of type
\(
    \mathsf B_\ell
\). Let
\[
    \widetilde \Phi
    = \Phi \cup \{
        \pm 2 \mathrm e_i \mid 1 \leq i \leq \ell
    \},
\]
it is a root system of type
\(
    \mathsf{BC}_\ell
\). Suppose that we have subgroups
\(
    G_{2 \alpha} \leq G_\alpha
\) for short
\(
    \alpha \in \Phi
\) satisfying the following additional conditions involving a commutative unital ring \(K\):
\begin{enumerate}[label = {(C\arabic*)}, start = 6]

    \item \label{bc-comm-rel}
    \(
        [G_\alpha, G_\beta] \leq \prod_{ \substack{
            i \alpha + j \beta \in \widetilde \Phi\\
            i, j \in \{ 1, 2, \ldots \}
        } } G_{i \alpha + j \beta}
    \) for all non-antiparallel
    \(
        \alpha, \beta \in \widetilde \Phi
    \).

    \item \label{k-act} The multiplicative monoid
    \(
        K^\bullet
    \) acts on \(G_\alpha\) from the right by
    \(
        (g, k) \mapsto g \cdot_\alpha k
    \), this is a \(K\)-module structure unless \(\alpha\) is short in \(\Phi\) and a \(2\)-step nilpotent \(K\)-module structure otherwise (with the nilpotent filtration
    \(
        G_{2 \alpha}
    \)).

    \item \label{homo-comm} If
    \(
        \alpha, \beta \in \widetilde \Phi
    \) are linearly independent,
    \(
        x \in G_\alpha
    \),
    \(
        y \in G_\beta
    \), and
    \(
        [x, y] = \prod_{ \substack{
            i \alpha + j \beta \in \widetilde \Phi\\
            i, j \in \{ 1, 2, \ldots \}
        } } z_{ij}
    \) for
    \(
        z_{ij} \in G_{i \alpha + j \beta}
    \), then
    \[
        [x \cdot_\alpha k_x, y \cdot_\beta k_y]
        = \prod_{ \substack{
            i \alpha + j \beta \in \widetilde \Phi\\
            i, j \in \{ 1, 2, \ldots \}
        } } z_{ij}
            \cdot_{i \alpha + j \beta} k_x^i k_y^j
    \]
    for all
    \(
        k_x, k_y \in K
    \).

\end{enumerate}
Here \ref{bc-comm-rel} is called \textit{\(\mathsf{BC}_\ell\)-commutator relations}, see \cite[definition 3.2]{st-jordan} for the case of arbitrary generalized root systems. This condition implies that \(G_\alpha\) is abelian for long
\(
    \alpha \in \Phi
\) and \(2\)-step nilpotent with the nilpotent filtration
\(
    G_{2 \alpha} \leq G_\alpha
\) for short
\(
    \alpha \in \Phi
\). The condition \ref{homo-comm} is independent of the choice of \(z_{ij}\) by \ref{k-act} and lemma \ref{unip-subgr}. Clearly, \ref{bc-comm-rel}--\ref{homo-comm} are satisfied by a Steinberg group of a graded odd form \(K\)-algebra.

The conditions \ref{bc-comm-rel}--\ref{homo-comm} imply that \(R_{ij}\) are \(K\)-modules with
\(
    T_{ij}(ak)
    = T_{ij}(a) \cdot_{\mathrm e_j - \mathrm e_i} k
\),
\(
    \Delta^0_i
\) are \(2\)-step nilpotent \(K\)-modules with the nilpotent filtration
\(
    \mathcal D_i \leq \Delta^0_i
\) (corresponding to
\(
    G_{2 \mathrm e_i}
\)),
\(
    T_i(u \cdot k) = T_i(u) \cdot_{\mathrm e_i} k
\), and
\(
    T_i(kv) = T_i(v) \cdot_{2 \mathrm e_i} k
\) for
\(
    v \in \mathcal D_i
\). Also,
\begin{enumerate}[label = {(R\arabic*)}, start = 31]

    \item \label{cir-fil}
    \(
        \mathcal D_i \circ \Delta^0_j = 0
    \);

    \item \label{fil-str}
    \(
        R_{-i, j} * R_{ji} \subseteq \mathcal D_i
    \);

    \item \label{dot-fil}
    \(
        \mathcal D_i \cdot R_{ij}
        \subseteq \mathcal D_j
    \);

    \item \label{inv-lin}
    \(
        \inv{ka} = k \inv a
    \) for
    \(
        a \in R_{ij}
    \);

    \item \label{mul-lin}
    \(
        (ka) b = k(ab) = a(kb)
    \) for
    \(
        a \in R_{ij}
    \),
    \(
        b \in R_{jk}
    \);

    \item \label{str-lin}
    \(
        ka * b = k(a * b) = a * kb
    \) for
    \(
        a \in R_{-i, j}
    \),
    \(
        b \in R_{ji}
    \);

    \item \label{cir-lin}
    \(
        (u \cdot k) \circ v
        = (u \circ v) k
        = u \circ (v \cdot k)
    \) for
    \(
        u \in \Delta^0_i
    \),
    \(
        v \in \Delta^0_j
    \);

    \item \label{tri-lin}
    \(
        (u \cdot k) \triangleleft a
        = k^2 (u \triangleleft a)
    \) and
    \(
        u \triangleleft ka = k(u \triangleleft a)
    \) for
    \(
        u \in \Delta^0_i
    \),
    \(
        a \in R_{ij}
    \);

    \item \label{tri-lin-fil}
    \(
        kv \triangleleft a = k(v \triangleleft a)
    \) for
    \(
        v \in \mathcal D_i
    \),
    \(
        a \in R_{ij}
    \);

    \item \label{dot-lin}
    \(
        (u \cdot k) \cdot a
        = u \cdot ka
        = (u \cdot a) \cdot k
    \) for
    \(
        u \in \Delta^0_i
    \),
    \(
        a \in R_{ij}
    \);

    \item \label{dot-lin-fil}
    \(
        kv \cdot a = k(v \cdot a)
    \) for
    \(
        v \in \mathcal D_i
    \),
    \(
        a \in R_{ij}
    \);

\end{enumerate}
where all indices are non-zero and with distinct absolute values. We say that an abstract partial graded odd form ring
\(
    (R_{ij}, \Delta^0_i)_{ij}
\) is a \textit{partial graded odd form \(K\)-algebra} if \(R_{ij}\) are \(K\)-modules,
\(
    \Delta^0_i
\) are \(2\)-step nilpotent \(K\)-modules, and \ref{cir-fil}--\ref{dot-lin-fil} hold.

It turns out that not every partial graded odd form algebra is a part of a graded odd form algebra. Let
\(
    (R_{ij}, \Delta^0_i)_{ij}
\) be a partial graded odd form ring satisfying \ref{long-idem} and \ref{short-idem}. If it completes to a graded odd form algebra, then the following \textit{associativity conditions} hold for various parentheses placements in the formal product
\(
    abcde
\) for
\(
    a \in R_{ij}
\),
\(
    b \in R_{jk}
\),
\(
    c \in R_{kl}
\),
\(
    d \in R_{lm}
\),
\(
    e \in R_{mn}
\):
\begin{enumerate}[label = {(A\arabic*)}]

    \item \label{ass-law-l}
    \(
        a((b(cd))e) = ((ab)c)(de)
    \), where
    \(
        |i| = |m|
    \),
    \(
        |j| = |l|
    \),
    \(
        |k| = |n|
    \) are distinct and non-zero;

    \item \label{ass-law-r}
    \(
        (ab)(c(de)) = (a((bc)d))e
    \), where
    \(
        |i| = |l|
    \),
    \(
        |j| = |n|
    \),
    \(
        |k| = |m|
    \) are distinct and non-zero;

    \item \label{ass-law-sym}
    \(
        ((a(bc))d)e = a(b((cd)e))
    \), where
    \(
        |i| = |k|
    \),
    \(
        |j| = |m|
    \),
    \(
        |l| = |n|
    \) are distinct and non-zero;

    \item \label{ass-law-bal}
    \(
        ((ab)(cd))e = a((bc)(de))
    \), where
    \(
        |i| = |l|
    \),
    \(
        |j| = |m|
    \),
    \(
        |k| = |n|
    \) are distinct and non-zero.

\end{enumerate}
It is easy to see that these are the only possible nontrivial generalized associative laws with \(5\) variables such that all partial products take values in \(R_{ij}\) for distinct non-zero \(|i|\) and \(|j|\). There are no such nontrivial laws with smaller number of variables.

\begin{lemma} \label{ass-cond}
    The conditions \ref{ass-law-l} and \ref{ass-law-r} are equivalent to each other and to the conjuction of \ref{ass-law-sym}, \ref{ass-law-bal}.
\end{lemma}
\begin{proof}
    Since the involution swaps \ref{ass-law-l} and \ref{ass-law-r}, they are equivalent. Consider parentheses placements in the formal product
    \(
        abcdef
    \) for
    \(
        a \in R_{ij}
    \),
    \(
        b \in R_{jk}
    \),
    \(
        c \in R_{kl}
    \),
    \(
        d \in R_{km}
    \),
    \(
        e \in R_{mn}
    \),
    \(
        f \in R_{no}
    \). The conditions \ref{ass-law-sym} and \ref{ass-law-bal} imply \ref{ass-law-l} since
    \[
        a(((bc)(de))f)
        = a(b((cd)(ef)))
        = ((a(bc))d)(ef)
    \]
    for distinct non-zero
    \(
        |i| = |k| = |n|
    \),
    \(
        |j| = |m|
    \),
    \(
        |l| = |o|
    \).

    Conversely, assume \ref{ass-law-l} and \ref{ass-law-r}. Then \ref{ass-law-bal} follows from
    \[
        ((ab)((cd)e))f
        = ((ab)c)(d(ef))
        = a((b(cd))(ef))
    \]
    for distinct non-zero
    \(
        |i| = |m|
    \),
    \(
        |j| = |l| = |n|
    \),
    \(
        |k| = |o|
    \). The conditions \ref{ass-law-sym} now follows from
    \[
        (ab)(c((de)f))
        = (a((bc)(de)))f
        = (((ab)(cd))e)f
    \]
    for distinct non-zero
    \(
        |i| = |l|
    \),
    \(
        |j| = |m| = |o|
    \),
    \(
        |k| = |n|
    \).
\end{proof}

\begin{theorem} \label{ass-forced}
    The associativity conditions hold if
    \(
        \ell \geq 4
    \) or
    \(
        R_{ij} = \langle
            \Delta^0_{-i} \circ \Delta^0_j
        \rangle
    \) for distinct non-zero \(|i|\) and \(|j|\).
\end{theorem}
\begin{proof}
    By lemma \ref{ass-cond} it suffices to prove \ref{ass-law-l}. Let
    \(
        a \in R_{ij}
    \),
    \(
        b \in R_{jk}
    \),
    \(
        c \in R_{kl}
    \),
    \(
        d \in R_{lm}
    \),
    \(
        e \in R_{mn}
    \) for distinct non-zero
    \(
        |i| = |m|
    \),
    \(
        |j| = |l|
    \),
    \(
        |k| = |n|
    \). In the case
    \(
        \ell \geq 4
    \) choose a non-zero index \(o\) distinct from all of them. Since \ref{ass-law-l} is polyadditive on its variables, we may assume that \(c = xy\) for
    \(
        x \in R_{ko}
    \) and
    \(
        y \in R_{ol}
    \). Then
    \begin{align*}
        a((b(cd))e)
        &= a((b((xy)d))e)
        = a((b(x(yd)))e)
        = a(((bx)(yd))e)
        = a((bx)((yd)e))
        \\
        &= (a(bx))(y(de))
        = (((ab)x)y)(de)
        = ((ab)(xy))(de)
        = ((ab)c)(de)
    \end{align*} by \ref{mul-ass}.

    If \(\ell = 3\) and
    \(
        R_{ij} = \langle
            \Delta^0_{-i} \circ \Delta^0_j
        \rangle
    \) for non-zero distinct \(|i|\) and \(|j|\), then we may assume that
    \(
        c = u \circ v
    \) for
    \(
        u \in \Delta^0_{-k}
    \) and
    \(
        v \in \Delta^0_l
    \). Then
    \begin{align*}
        a((b(cd))e)
        &= a((b((u \circ v)d))e)
        = a((b(u \circ (v \cdot d)))e)
        = a(((u \cdot \inv{\,b\,}) \circ (v \cdot d))e)
        = a((u \cdot \inv{\,b\,})
            \circ ((v \cdot d) \cdot e))
        \\
        &= ((u \cdot \inv{\,b\,}) \cdot \inv a)
            \circ (v \cdot de)
        = ((u \cdot \inv{(ab)}) \circ v)(de)
        = ((ab)(u \circ v))(de)
        = ((ab)c)(de)
    \end{align*} by \ref{inv-mul}, \ref{inv-cir}, \ref{cir-dot}, and \ref{dot-ass}.
\end{proof}

Now we give examples showing the necessity of the associativity conditions for \(\ell = 3\).

\begin{example}
    Let
    \(
        L = \mathfrak{sl}(3, \mathbb O)
    \) be the simple real Lie algebra constructed in \cite[theorem 5.1]{magic-square}. It is naturally graded by
    \(
        \mathsf A_3
    \) with the non-diagonal components
    \(
        e_{ij} \mathbb O
    \) for distinct
    \(
        1 \leq i, j \leq 3
    \) such that
    \begin{align*}
        [e_{ij} x, e_{jk} y] &= e_{ik} xy
        \text{ for distinct } i, j, k > 0;
        &
        [e_{ij} x, e_{kl} y] &= 0
        \text{ for } j \neq k \text{ and } i \neq l.
    \end{align*}
    It follows that
    \(
        R_{ij} = e_{ij} \mathbb O
    \) for \(ij > 0\),
    \(
        R_{ij} = 0
    \) for \(ij < 0\),
    \(
        \Delta^0_i = \dot 0
    \) are components of a partial graded odd form \(\mathbb R\)-algebra, where
    \begin{align*}
        \inv{(e_{ij} x)} &= e_{-j, -i} \inv x,
        &
        (e_{ij} x) (e_{jk} y) &= e_{ik} xy.
    \end{align*}
    Moreover, there is a corresponding group
    \(
        G = \Aut^\circ(L)
    \) with subgroups satisfying \ref{comm-rel}--\ref{homo-comm},
    \(
        T_{ij}(e_{ij} x)
        = \exp(\mathrm{ad}_{e_{ij} x})
    \) for \(i, j > 0\).
\end{example}

\begin{example}
    Let
    \(
        L = \mathfrak{sp}(6, \mathbb O)
    \) be the simple real Lie algebra constructed in \cite[theorem 5.1]{magic-square} (we choose the split symplectic form with the matrix
    \(
        J = \sum_i \eps_i e_{-i, i}
    \), where
    \(
        \eps_i = 1
    \) for \(i > 0\) and
    \(
        \eps_i = -1
    \) for \(i < 0\)). It is naturally graded by
    \(
        \mathsf C_3
    \) with the non-diagonal components
    \(
        \{
            e_{ij} x
            - \eps_i \eps_j e_{-j, -i} \inv x
        \mid
            x \in \mathbb O
        \}
    \) for distinct non-zero \(i\) and \(j\) (possibly \(i = -j\), in this case we may consider only
    \(
        x \in \mathbb R
    \)). Here
    \begin{align*}
        [
            e_{ij} x
            - \eps_i \eps_j e_{-j, -i} \inv x,
            e_{kl} y
            - \eps_k \eps_l e_{-l, -k} \inv y
        ]
        &= 0
        \text{ for } j \neq k \neq -i \neq -l \neq j,
        \\
        [
            e_{ij} x
            - \eps_i \eps_j e_{-j, -i} \inv x,
            e_{jk} y
            - \eps_j \eps_k e_{-k, -j} \inv y
        ]
        &= e_{ik} xy
        - \eps_i \eps_k e_{-k, -i} \inv{(xy)}
        \text{ for } k \neq i.
    \end{align*}

    It follows that
    \(
        R_{ij} = e_{ij} \mathbb O
    \) for distinct non-zero \(i\), \(j\) and
    \(
        \Delta^0_i
        = \mathcal D_i
        = e_{-i, i} \mathbb R
    \) are components a partial graded odd form \(\mathbb R\)-algebra, where
    \begin{align*}
        \inv{(e_{ij} x)}
        &= \eps_i \eps_j e_{-j, -i} \inv x,
        &
        e_{-i, i} a \circ e_{-j, j} b &= 0,
        \\
        (e_{ij} x) (e_{jk} y) &= e_{ik} xy
        \text{ for } i \neq \pm k,
        &
        e_{-i, i} a \triangleleft e_{ij} x
        &= e_{-i, j} a^2 x,
        \\
        e_{-i, j} x * e_{ji} y
        &= e_{-i, i} (xy + \inv{(xy)}),
        &
        e_{-i, i} a \cdot e_{ij} x
        &= \eps_i \eps_j e_{-j, j} \inv x a x.
    \end{align*}
    There exists a corresponding group
    \(
        G = \Aut^\circ(L)
    \) with subgroups satisfying \ref{comm-rel}--\ref{homo-comm},
    \(
        T_{ij}(
            e_{ij} x
            - \eps_i \eps_j e_{-j, -i} \inv x
        )
        = \exp(\mathrm{ad}_{
            e_{ij} x
            - \eps_i \eps_j e_{-j, -i} \inv x
        })
    \),
    \(
        T_j(e_{-j, j} a)
        = \exp(\mathrm{ad}_{e_{-j, j} a})
    \).
\end{example}

\begin{example}
    Let
    \(
        G = \glin(4, D)
    \) with root elements \(t_{ij}(a)\) for some associative unital ring \(D\). Its root subgroups are parameterized by a root system of type
    \(
        \mathsf A_3
    \). Since
    \(
        \mathsf A_3 = \mathsf D_3
    \), there is a corresponding partial graded odd form \(\mathbb Z\)-algebra with \(\ell = 3\),
    \(
        R_{ij} = e_{ij} D
    \),
    \(
        \Delta^0_i = \dot 0
    \) such that
    \begin{align*}
        T_{ij}(e_{ij} a)
        = T_{-j, -i}(e_{-j, -i} a)
        &= t_{ij}(a)
        \text{ for } 1 \leq i, j \leq 3;
        &
        \inv{ (e_{ij} a) } &= -e_{-j, -i} a;
        \\
        T_{-i, j}(e_{-i, j} a) &= t_{k4}(a)
        \text{ for } \{ i, j, k \} = \{ 1, 2, 3 \};
        &
        (e_{ij} a) (e_{jk} b) &= e_{ik} ab
        \text{ for } ijk > 0;
        \\
        T_{i, -j}(e_{i, -j} a) &= t_{4k}(a)
        \text{ for } \{ i, j, k \} = \{ 1, 2, 3 \};
        &
        (e_{ij} a) (e_{jk} b) &= -e_{ik} ba
        \text{ for } ijk < 0.
    \end{align*}
    It satisfies the associativity conditions if and only if \(D\) is commutative. Namely, the instance of \ref{ass-law-l} for
    \(
        a \in R_{1, -2}
    \),
    \(
        b \in R_{-2, 3}
    \),
    \(
        c \in R_{32}
    \),
    \(
        d \in R_{21}
    \),
    \(
        e \in R_{13}
    \) implies that
    \(
        vzuyx = yxzuv
    \) for any
    \(
        x, y, z, u, v \in D
    \), i.e. \(D\) is commutative.
\end{example}

\section{Construction of the involution algebra}

In this and the next sections
\(
    (R_{ij}, \Delta^0_i)_{ij}
\) is a partial graded odd form \(K\)-algebra satisfying \ref{long-idem}, \ref{short-idem}, and the associativity conditions.

Let
\(
    R_{0i} = \Delta^0_i / \mathcal D_i
\) for each \(i \neq 0\), it is a \(K\)-module. The projection map
\(
    \Delta^0_i \to R_{0i}
\) is denoted by \(\pi\). Let
\(
    \inv{(-)} \colon R_{i0} \to R_{0, -i}
\) be arbitrary \(K\)-module isomorphisms for \(i \neq 0\) and some \(K\)-modules \(R_{i0}\) with inverse homomorphisms denoted again by \(\inv{(-)}\). It follows that there are unique maps
\begin{align*}
    ({-})({=})
    &\colon R_{0i} \times R_{ij}
    \to R_{0j},
    &
    ({-})({=})
    &\colon R_{ji} \times R_{i0}
    \to R_{j0},
    \\
    ({-})({=})
    &\colon R_{i0} \times R_{0j}
    \to R_{ij},
    &
    ({-}) * ({=})
    &\colon R_{-i, 0} \times R_{0i}
    \to \Delta^0_i
\end{align*}
for distinct non-zero \(|i|\), \(|j|\) such that the following identities hold, where all indices are with distinct absolute values (but possible zero unless otherwise stated):
\begin{enumerate}[label = {(Z\arabic*)}]

    \item \label{pi-lin}
    \(
        \pi(u \dotplus v) = \pi(u) + \pi(v)
    \),
    \(
        \pi(u \cdot k) = \pi(u) k
    \) for
    \(
        u, v \in \Delta^0_i
    \), \(k \in K\), \(i \neq 0\);

    \item \label{pi-aug}
    \(
        \pi(\mathcal D_i) = 0
    \) for \(i \neq 0\);

    \item \label{pi-dot}
    \(
        \pi(u) a = \pi(u \cdot a)
    \) for
    \(
        u \in \Delta^0_i
    \),
    \(
        a \in R_{ij}
    \), \(i, j \neq 0\) (the left hand side is well-defined by this equation, see \ref{dot-dis-l} and \ref{dot-fil});

    \item \label{mul-pi}
    \(
        \inv{\pi(u)} \pi(v) = u \circ v
    \) for
    \(
        u \in \Delta^0_i
    \),
    \(
        v \in \Delta^0_j
    \), \(i, j \neq 0\) (the left hand side is well-defined by this equation, see \ref{cir-dis-l}, \ref{cir-dis-r}, \ref{inv-cir}, \ref{cir-fil});

    \item \label{inv-mul-2}
    \(
        \inv{(ab)} = \inv{\,b\,} \inv a
    \) for
    \(
        a \in R_{ij}
    \),
    \(
        b \in R_{jk}
    \) (this is a definition for \(|k| = 0\) and a corollary of \ref{inv-cir} for \(|j| = 0\));

    \item \label{mul-lin-2}
    \(
        (a, b) \mapsto ab
    \) is \(K\)-bilinear for
    \(
        a \in R_{ij}
    \),
    \(
        b \in R_{jk}
    \) (by \ref{cir-dis-l}, \ref{cir-dis-r}, \ref{dot-dis-l}, \ref{dot-dis-r}, \ref{cir-lin}, \ref{dot-lin});

    \item \label{mul-ass-2}
    \(
        (ab)c = a(bc)
    \) for
    \(
        a \in R_{ij}
    \),
    \(
        b \in R_{jk}
    \),
    \(
        c \in R_{kl}
    \) (by \ref{cir-dot}, \ref{dot-ass});

    \item \label{half-idem}
    \(
        R_{0i} = R_{0j} R_{ji} + R_{0, -j} R_{-j, i}
    \) for non-zero \(i\) and \(j\) with \(|i| \neq |j|\) (by \ref{short-idem});

    \item \label{str-pi}
    \(
        \inv{\pi(u)} * \pi(v)
        = v \dotplus u \dotminus v \dotminus u
    \) for
    \(
        u, v \in \Delta^0_i
    \), \(i \neq 0\) (by definition);

    \item \label{str-lin-2}
    \(
        (a, b) \mapsto a * b \in \mathcal D_i
    \) for
    \(
        a \in R_{-i, j}
    \),
    \(
        b \in R_{ji}
    \), \(i \neq 0\) is \(K\)-bilinear;

    \item \label{str-inv-2}
    \(
        a * b \dotplus \inv{\,b\,} * \inv a = \dot 0
    \) for
    \(
        a \in R_{-i, j}
    \),
    \(
        b \in R_{ji}
    \), \(i \neq 0\);

    \item \label{str-inv-sq}
    \(
        \inv a * a = \dot 0
    \) for
    \(
        a \in R_{0i}
    \), \(i \neq 0\);

    \item \label{str-ass-2}
    \(
        ab * c = a * bc
    \) for
    \(
        a \in R_{-i, j}
    \),
    \(
        b \in R_{jk}
    \),
    \(
        c \in R_{ki}
    \), \(i \neq 0\) (by \ref{dot-cen}).

\end{enumerate}

Now we may eliminate the operation
\(
    ({-}) \circ ({=})
\) from \ref{mul-cen}--\ref{dot-ass}. Here the indices are non-zero and with distinct absolute values:

\begin{enumerate}[label = {(Z\arabic*)}, start = 13]

    \item \label{tri-dis-l-2}
    \(
        (u \dotplus v) \triangleleft a
        = u \triangleleft a
        - \inv{ \pi(u) } (\pi(v) a)
        + v \triangleleft a
    \) for
    \(
        u, v \in \Delta^0_i
    \),
    \(
        a \in R_{ij}
    \);

    \item \label{inv-tri-2}
    \(
        \inv a (u \triangleleft b)
        + \inv{ (\pi(u) a) } (\pi(u) b)
        + \inv{(u \triangleleft a)} b
        = 0
    \) for
    \(
        u \in \Delta^0_i
    \),
    \(
        a \in R_{ij}
    \),
    \(
        b \in R_{ik}
    \).

\end{enumerate}

We construct \(R_{ij}\) for \(|i| = |j|\) as abstract \(K\)-modules using generators and relations. The generators are formal products \(ab\) for
\(
    a \in R_{ik}
\),
\(
    b \in R_{kj}
\), and \(
    k \notin \{ i, -i \}
\). The relations are
\begin{itemize}

    \item
    \(
        (a, b) \mapsto ab
    \) are bilinear;

    \item
    \(
        (ab)c = a(bc)
    \),
    \(
        a \in R_{ik}
    \),
    \(
        b \in R_{kl}
    \),
    \(
        c \in R_{lj}
    \),
    \(
        k, l \notin \{ i, -i \}
    \), \(|k| \neq |l|\).

\end{itemize}
For \(R_{-i, i}\), \(i \neq 0\), we add extra generators \(\rho(u)\) for
\(
    u \in \Delta^0_i
\) and extra relations
\begin{itemize}

    \item \(
        \rho(u \dotplus v)
        = \rho(u) - \inv{\pi(u)} \pi(v) + \rho(v)
    \) for
    \(
        u, v \in \Delta^0_i
    \);

    \item
    \(
        \rho(u \cdot k) = \rho(u) k^2
    \) and
    \(
        \rho(kv) = \rho(v) k
    \) for
    \(
        u \in \Delta^0_i
    \),
    \(
        v \in \mathcal D_i
    \), \(k \in K\);

    \item
    \(
        \rho(u \cdot a)
        = \inv a (u \triangleleft a)
    \) for
    \(
        u \in \Delta^0_j
    \),
    \(
        a \in R_{ji}
    \),
    \(
        j \notin \{ 0, i, -i \}
    \);

    \item
    \(
        \rho(a * b) = ab - \inv{\,b\,} \inv a
    \) for
    \(
        a \in R_{-i, j}
    \),
    \(
        b \in R_{ji}
    \),
    \(
        j \notin \{ i, -i \}
    \).

\end{itemize}

The next lemma is usually applied to the groups \(R_{ij}\),
\(
    R_{ik} \oplus R_{i, -k}
\),
\(
    R_{kj} \oplus R_{k, -j}
\), and
\(
    R_{kl}
    \oplus R_{-k, l}
    \oplus R_{k, -l}
    \oplus R_{-k, -l}
\) for \(k, l \neq 0\).

\begin{lemma} \label{mul-cons}
    Let \(A_{ij}\) be abelian groups for
    \(
        0 \leq i < j \leq 4
    \) unless
    \(
        (i, j) = (1, 3)
    \) and
    \(
        \mu_{ijk}
        \colon A_{ij} \times A_{jk}
        \to A_{ik}
    \) be biadditive maps for
    \(
        (i, j, k) \in \{
            (0, 1, 2),
            (0, 2, 3),
            (0, 3, 4),
            (0, 1, 4),
            (1, 2, 4),
            (2, 3, 4)
        \}
    \). Suppose that
    \(
        A_{02} = \langle
            \mu_{02}(A_{01}, A_{12})
        \rangle
    \),
    \(
        A_{24} = \langle
            \mu_{24}(A_{23}, A_{34})
        \rangle
    \), and
    \(
        \mu_{014}(a, \mu_{124}(b, \mu_{234}(c, d)))
        = \mu_{034}(\mu_{023}(\mu_{012}(a, b), c), d)
    \). Then there is unique biadditive map
    \(
        \mu_{024}
        \colon A_{02} \times A_{24}
        \to A_{024}
    \) such that
    \(
        \mu_{024}(\mu_{012}(a, b), c)
        = \mu_{014}(a, \mu_{124}(b, c))
    \) and
    \(
        \mu_{034}(\mu_{023}(a, b), c)
        = \mu_{024}(a, \mu_{234}(b, c))
    \).
\end{lemma}
\begin{proof}
    Take
    \(
        a \in A_{02}
    \) and
    \(
        b \in A_{24}
    \). Choose decompositions
    \(
        a = \sum_{s = 1}^N \mu_{012}(x_s, y_s)
    \) and
    \(
        b = \sum_{t = 1}^M \mu_{234}(z_t, w_t)
    \). Writing
    \(
        \mu_{ijk}(p, q)
    \) as \(pq\), we have
    \[
        \sum_s x_s (y_s b)
        = \sum_{s, t} x_s (y_s (z_t w_t))
        = \sum_{s, t} ((x_s y_s) z_t) w_t
        = \sum_t (a z_t) w_t.
    \]
    It follows that the common value
    \(
        \mu_{024}(a, b)
    \) of these expressions is independent on the decompositions. The resulting map \(\mu_{024}\) satisfies the necessary identities and is clearly unique.
\end{proof}

\begin{lemma} \label{tri-dot-2}
    If
    \(
        u \in \Delta^0_i
    \),
    \(
        a \in R_{ij}
    \),
    \(
        b \in R_{jk}
    \), where the indices are non-zero and with distinct absolute values, then
    \[
        (u \cdot a) \triangleleft b
        = \inv a ((u \triangleleft a) b).
    \]
\end{lemma}
\begin{proof}
    By \ref{tri-dot}, \ref{pi-dot}, \ref{tri-dis-l-2}, \ref{inv-tri-2} the left hand side is
    \[
        \inv{ ((\dotminus u) \triangleleft a) } (ab)
        = \inv{ (
            - u \triangleleft a
            - \inv{\pi(u)} \pi(u \cdot a)
        ) } (ab)
        = \inv a (u \triangleleft ab)
        - (\inv{ \pi(u \cdot a) } \pi(u)) (ab)
        + \inv{ \pi(u \cdot a) } \pi(u \cdot ab).
    \]
    This coincides with the right hand side by \ref{tri-ass}, \ref{pi-dot}, and \ref{mul-ass-2}.
\end{proof}

\begin{lemma} \label{pres-ring}
    Let \(i\), \(j\) be indices such that
    \(
        |i| = |j| \in \{ 0, 1 \}
    \). Then \(R_{ij}\) is generated by formal products \(ab\),
    \(
        a \in R_{ik}
    \),
    \(
        b \in R_{kj}
    \), \(|k| = 2\) with the only relations
    \begin{itemize}

        \item
        \(
            (a, b) \mapsto ab
        \) are biadditive;

        \item
        \(
            ((ab)c)d = a(b(cd))
        \),
        \(
            a \in R_{ik}
        \),
        \(
            b \in R_{kl}
        \),
        \(
            c \in R_{lm}
        \),
        \(
            d \in R_{mj}
        \),
        \(
            |k| = |m| = 2
        \),
        \(
            |l| = 3
        \).

    \end{itemize}
    In the case
    \(
        i = -j \neq 0
    \) there are additional generators \(\rho(u)\) for
    \(
        u \in \Delta^0_j
    \) and additional relations
    \begin{itemize}

        \item
        \(
            \rho(u \cdot k) = \rho(u) k^2
        \) and
        \(
            \rho(kv) = \rho(v) k
        \) for
        \(
            u \in \Delta^0_i
        \),
        \(
            v \in \mathcal D_i
        \), \(k \in K\);

        \item
        \(
            \rho(u \dotplus v \cdot a)
            = \rho(u)
            - (\inv{\pi(u)} \pi(v)) a
            + \inv a (v \triangleleft a)
        \) for
        \(
            u \in \Delta^0_j
        \),
        \(
            v \in \Delta^0_k
        \),
        \(
            a \in R_{kj}
        \), \(|k| = 2\);

        \item
        \(
            \rho(u \dotplus a * b)
            = \rho(u) + ab - \inv{\,b\,} \inv a
        \) for
        \(
            u \in \Delta^0_j
        \),
        \(
            a \in R_{ik}
        \),
        \(
            b \in R_{kj}
        \), \(|k| = 2\).

    \end{itemize}
\end{lemma}
\begin{proof}
    Let \(R'_{ij}\) be the group with presentation from the statement. First of all let us check that
    \(
        ((ab)c)d = a(b(cd))
    \) holds in \(R'_{ij}\) for
    \(
        a \in R_{ik}
    \),
    \(
        b \in R_{kl}
    \),
    \(
        c \in R_{lm}
    \),
    \(
        d \in R_{mj}
    \),
    \(
        |k| = |m| = 2
    \),
    \(
        |l| \notin \{ |i|, 2, 3 \}
    \). Indeed, without loss of generality
    \(
        b = xy
    \) for
    \(
        x \in R_{k, \pm 3}
    \) and
    \(
        y \in R_{\pm 3, l}
    \), so
    \begin{align*}
        ((a(xy))c)d
        = (((ax)y)c)d
        = ((ax)(yc))d
        = a(x((yc)d))
        = a(x(y(cd)))
        = a((xy)(cd)).
    \end{align*}
    Now by lemma \ref{mul-cons} there is unique multiplication
    \(
        R_{ik} \times R_{kj} \to R'_{ij}
    \) for
    \(
        |k| \notin \{ |i|, 2 \}
    \) such that
    \(
        (ab) c = a (bc)
    \) for
    \(
        a \in R_{ik}
    \),
    \(
        b \in R_{k, \pm 2}
    \),
    \(
        c \in R_{\pm 2, j}
    \) and for
    \(
        a \in R_{i, \pm 2}
    \),
    \(
        b \in R_{\pm 2, k}
    \),
    \(
        c \in R_{kj}
    \). We check that
    \(
        (ab) c = a (bc)
    \) for
    \(
        a \in R_{ik}
    \),
    \(
        b \in R_{kl}
    \),
    \(
        c \in R_{lj}
    \), and distinct \(|i|\), \(|k|\), \(|l|\) such that
    \(
        |k|, |l| \neq 2
    \). Without loss of generality, \(b = xy\) for some
    \(
        x \in R_{k, \pm 2}
    \),
    \(
        y \in R_{\pm 2, l}
    \). Then
    \[
        (a (xy)) c
        = ((ax) y) c
        = (ax) (yc)
        = a (x (yc))
        = a ((xy) c).
    \]

    Now consider the case
    \(
        i = -j \in \{ -1, 1 \}
    \). By \ref{short-idem}
    \(
        \rho(u \dotplus v)
        = \rho(u) - \inv{\pi(u)} \pi(v) + \rho(v)
    \) for all
    \(
        u, v \in \Delta^0_j
    \). In order to check that
    \(
        \rho(a * b) = ab - \inv{\,b\,} \inv a
    \) for
    \(
        a \in R_{ik}
    \),
    \(
        b \in R_{kj}
    \),
    \(
        |k| \notin \{ 1, 2 \}
    \) we may assume that \(a = xy\) for
    \(
        x \in R_{il}
    \),
    \(
        y \in R_{lk}
    \), \(|l| = 2\). Then
    \[
        \rho(xy * b)
        = \rho(x * y b)
        = x (y b) - \inv{ (y b) } \inv x
        = (xy) b - \inv{\,b\,} \inv{ (xy) }.
    \]
    It remains to check that
    \(
        \rho(u \cdot a) = \inv a (u \triangleleft a)
    \) for
    \(
        u \in \Delta^0_k
    \),
    \(
        a \in R_{kj}
    \), \(|k| \geq 3\). Here it suffices to consider the cases
    \(
        u = w \cdot x
    \),
    \(
        w \in \Delta^0_l
    \),
    \(
        x \in R_{lk}
    \), \(|l| = 2\) and
    \(
        u = y * z
    \),
    \(
        y \in R_{-k, 2}
    \),
    \(
        z \in R_{2k}
    \). In the first case we have
    \[
        \rho((w \cdot x) \cdot a)
        = \rho(w \cdot xa)
        = \inv{ (xa) } (w \triangleleft xa)
        = \inv a (\inv{\,x\,} ((w \triangleleft x) a))
        = \inv a ((w \cdot x) \triangleleft a)
    \]
    using lemma \ref{tri-dot-2} and in the second case
    \[
        \rho((y * z) \cdot a)
        = \rho(\inv a y * za)
        = (\inv a y) (za) - \inv{ (za) } (\inv y a)
        = \inv a (y (za) - \inv z (\inv y a))
        = \inv a ((y * z) \triangleleft a).
    \qedhere\]
\end{proof}

The abelian group
\(
    R = \bigoplus_{i, j = -\ell}^\ell R_{ij}
\) has natural involution, it is given on generators of \(R_{ij}\) for \(|i| = |j|\) by
\begin{align*}
    \inv{ (ab) } &= \inv{\,b\,} \inv a,
    &
    \inv{ \rho(u) }
    = \rho(\dotminus u)
    = -\inv{\pi(u)} \pi(u) - \rho(u).
\end{align*}

\begin{lemma} \label{ring-struct}
    There are unique product maps
    \(
        R_{ij} \times R_{jk} \to R_{ik}
    \) for
    \(
        |j| \in \{ |i|, |k| \}
    \) and involution maps
    \(
        \inv{(-)} \colon R_{i, \pm i} \to R_{\mp i, -i}
    \) making
    \(
        R = \bigoplus_{i, j = -l}^l R_{ij}
    \) an associative \(K\)-algebra with involution. Also,
    \(
        u \triangleleft a = \rho(u) a
    \) for
    \(
        u \in \Delta^0_i
    \),
    \(
        a \in R_{ij}
    \),
    \(
        0 \neq |i| \neq |j| \neq 0
    \);
    \(
        \inv{\rho(u)} = -\inv{\pi(u)} \pi(u) - \rho(u)
    \) for
    \(
        u \in \Delta^0_i
    \), \(i \neq 0\); and
    \(
        R_{ij} = \langle
            R_{ik} R_{kj},
            R_{i, -k} R_{-k, j}
        \rangle
    \) for all \(i\), \(j\), and \(k \neq 0\).
\end{lemma}
\begin{proof}
    First of all, lemma \ref{tri-dot-2} holds in the case
    \(
        |j| = |k| \neq |i|
    \). Indeed, if
    \(
        b = xy
    \) for
    \(
        x \in R_{jl}
    \),
    \(
        y \in R_{lk}
    \),
    \(
        |l| \notin \{ |i|, |k| \}
    \), then
    \[
        (u \cdot a) \triangleleft xy
        = ((u \cdot a) \triangleleft x) y
        = (\inv a ((u \triangleleft a) x)) y
        = \inv a (((u \triangleleft a) x) y)
        = \inv a ((u \triangleleft a) (xy)).
    \]
    It follows that the involution
    \(
        \inv{(-)} \colon R_{ij} \to R_{-j, -i}
    \) for
    \(
        |i| = |j|
    \) may be constructed on generators by
    \(
        \inv{(ab)} = \inv{\,b\,} \inv a
    \) for
    \(
        a \in R_{ik}
    \),
    \(
        b \in R_{kj}
    \),
    \(
        |k| \neq |i|
    \) and
    \(
        \inv{\rho(u)} = -\inv{\pi(u)} \pi(u) - \rho(u)
    \) for
    \(
        u \in \Delta^0_j
    \),
    \(
        i = -j \neq 0
    \).

    We denote by
    \(
        \mathrm A[i, j, k, l]
    \) the identity
    \(
        (ab)c = a(bc)
    \) for
    \(
        a \in R_{ij}
    \),
    \(
        b \in R_{jk}
    \),
    \(
        c \in R_{kl}
    \). It holds if the absolute values of the indices are distinct except, possibly, \(|i| = |l|\). In general it suffices to check such identities on generators of the corresponding abelian groups.

    Let us construct multiplications
    \(
        R_{ij} \times R_{jk} \to R_{ik}
    \) for
    \(
        |i| = |j| \neq |k|
    \) satisfying
    \(
        \mathrm A[i, l, j, k]
    \) for
    \(
        |l| \notin \{ |i|, |k| \}
    \) and
    \(
        \rho(u) a = u \triangleleft a
    \) for
    \(
        u \in \Delta^0_j
    \),
    \(
        a \in R_{jk}
    \) in the case
    \(
        i = -j \neq 0
    \). Without loss of generality, \(|i| = 1\) and \(|k| = 3\) or \(i = 0\) and \(|k| = 1\). These two identities give a well-defined multiplication when applied to generators from lemma \ref{pres-ring} by \ref{ass-law-l} and other identities. The identities
    \(
        \mathrm A[i, l, j, k]
    \) for
    \(
        |l| \notin \{ |i|, |k|, 2 \}
    \) follow from
    \[
        ((xy)b)c
        = (x(yb))c
        = x((yb)c)
        = x(y(bc))
        = (xy)(bc)
    \]
    for any
    \(
        x \in R_{im}
    \),
    \(
        y \in R_{ml}
    \), \(|m| = 2\). The multiplications
    \(
        R_{ij} \times R_{jk} \to R_{ik}
    \) for
    \(
        |i| \neq |j| = |k|
    \) may be constructed using the involution.

    From \ref{long-idem}, \ref{short-idem}, lemma \ref{pres-ring} we easily get
    \(
        R_{ij} = \langle
            R_{ik} R_{kj},
            R_{i, -k} R_{-k, j}
        \rangle
    \) for
    \(
        |k| \notin \{ |i|, |j|, 0 \}
    \). Now we check
    \(
        \mathrm A[i, j, k, l]
    \) if no integer appears thrice in \(|i|\), \(|j|\), \(|k|\), \(|l|\):
    \begin{itemize}

        \item Let \(|i| = |j|\), \(|k|\), \(|l|\) be distinct. If they are non-zero, then without loss of generality \(|i| = 1\), \(|k| = 2\), \(|l| = 3\) and we apply \ref{ass-law-bal} to get
        \[
            (((xy)z)b)c
            = ((xy)(zb))c
            = x((yz)(bc))
            = (x(yz))(bc)
        \]
        for
        \(
            x \in R_{im}
        \),
        \(
            y \in R_{mn}
        \),
        \(
            z \in R_{nj}
        \), \(|m| = 2\), \(|n| = 3\). Otherwise
        \[
            ((xy)b)c
            = (x(yb))c
            = x((yb)c)
            = x(y(bc))
            = (xy)(bc)
        \]
        for
        \(
            x \in R_{im}
        \),
        \(
            y \in R_{mj}
        \),
        \(
            m \notin \{ 0, |i|, |j|, |k|, |l| \}
        \). The case of distinct \(|i|\), \(|j|\), \(|k| = |l|\) is symmetric.

        \item Let \(|i|\), \(|j| = |k|\), \(|l|\) be distinct. If they are non-zero, then without loss of generality \(|j| = 1\), \(|i| = 2\), \(|l| = 3\). By \ref{ass-law-bal} we have
        \[
            (a(x(yz)))c
            = ((ax)(yz))c
            = a((xy)(zc))
            = a(((xy)z)c)
        \]
        for
        \(
            x \in R_{jm}
        \),
        \(
            y \in R_{mn}
        \),
        \(
            z \in R_{nk}
        \), \(|m| = 3\), \(|n| = 2\). Otherwise the identity follows from
        \[
            (a(xy))c
            = ((ax)y)c
            = (ax)(yc)
            = a(x(yc))
            = a((xy)c)
        \]
        for
        \(
            x \in R_{jm}
        \),
        \(
            y \in R_{mk}
        \),
        \(
            m \notin \{ 0, |i|, |j|, |k|, |l| \}
        \).

        \item In the remaining case \(|i|\), \(|j|\), \(|k|\), \(|l|\) form two pairs of equal numbers and the identity follows from
        \[
            (a(xy))c
            = ((ax)y)c
            = (ax)(yc)
            = a(x(yc))
            = a((xy)c)
        \]
        for
        \(
            a \in R_{ij}
        \),
        \(
            x \in R_{jm}
        \),
        \(
            y \in R_{mk}
        \),
        \(
            c \in R_{kl}
        \),
        \(
            |m| \notin \{ 0, |i|, |j|, |k|, |l| \}
        \).

    \end{itemize}

    We are ready to construct the multiplications
    \(
        R_{ij} \times R_{jk} \to R_{ik}
    \) for
    \(
        |i| = |j| = |k|
    \). Without loss of generality,
    \(
        |i| \in \{ 0, 1 \}
    \). By lemma \ref{mul-cons} such a map is unique with the properties
    \(
        \mathrm A[i, \pm 2, j, k]
    \) and
    \(
        \mathrm A[i, j, \pm 3, k]
    \) since
    \[
        ((ab)c)d
        = (a(bc))d
        = a((bc)d)
        = a(b(cd))
    \]
    for
    \(
        a \in R_{i, \pm 2}
    \),
    \(
        b \in R_{\pm 2, j}
    \),
    \(
        c \in R_{j, \pm 3}
    \),
    \(
        d \in R_{\pm 3, k}
    \). If \(|l| \neq |i|\), then
    \(
        \mathrm A[i, l, j, k]
    \) follows from
    \[
        (ab)(cd)
        = ((ab)c)d
        = (a(bc))d
        = a((bc)d)
        = a(b(cd))
    \]
    for
    \(
        a \in R_{il}
    \),
    \(
        b \in R_{lj}
    \),
    \(
        c \in R_{j, \pm 3}
    \),
    \(
        d \in R_{\pm 3, k}
    \). By symmetry,
    \(
        \mathrm A[i, j, m, k]
    \) holds for \(|m| \neq |i|\).

    The constructed multiplication on \(R\) clearly satisfies
    \(
        \inv{(ab)} = \inv{\,b\,} \inv a
    \) for all \(a\) and \(b\). The identities
    \(
        \mathrm A[i, j, k, l]
    \) for
    \(
        |i| = |j| = |k|
    \) or
    \(
        |j| = |k| = |l|
    \) follow from
    \[
        (a(xy))c
        = ((ax)y)c
        = (ax)(yc)
        = a(x(yc))
        = a((xy)c)
    \]
    for
    \(
        x \in R_{jm}
    \),
    \(
        y \in R_{mk}
    \),
    \(
        |m| \notin \{ 0, |i|, |j|, |k|, |l| \}
    \).

    It remains to check that \(R_{ij}\) is generated by
    \(
        R_{i, \pm k} R_{\pm k, j}
    \) for any \(k \neq 0\):
    \begin{itemize}

        \item If
        \(
            |i| = |k| \neq |j|
        \), then take \(l\) such that
        \(
            |l| \notin \{ |i|, |j|, 0 \}
        \). We have
        \begin{align*}
            R_{ij}
            &= \langle
                R_{il} R_{lj},
                R_{i, -l} R_{-l, j}
            \rangle \\
            &= \langle
                R_{il} R_{lk} R_{kj},
                R_{il} R_{l, -k} R_{-k, j},
                R_{i, -l} R_{-l, k} R_{kj},
                R_{i, -l} R_{-l, -k} R_{-k, j}
            \rangle \\
            &= \langle
                R_{ik} R_{kj},
                R_{i, -k} R_{-k, j}
            \rangle.
        \end{align*}

        \item If \(|k| = |j|\), then take \(l\) such that
        \(
            |l| \notin \{ |i|, |j|, 0 \}
        \). We have
        \begin{align*}
            R_{ij}
            &= \langle
                R_{il} R_{lj},
                R_{i, -l} R_{-l, j}
            \rangle \\
            &= \langle
                R_{ik} R_{kl} R_{lj},
                R_{i, -k} R_{-k, l} R_{lj},
                R_{ik} R_{k, -l} R_{-l, j},
                R_{i, -k} R_{-k, -l} R_{-l, j}
            \rangle \\
            &= \langle
                R_{ik} R_{kj},
                R_{i, -k} R_{-k, j}
            \rangle.
        \qedhere\end{align*}

    \end{itemize}
\end{proof}

\section{Construction of the odd form parameter}

In this section we everywhere implicitly use lemma \ref{ring-struct}. It is easy to check that there are unique \(K\)-module homomorphisms
\(
    \phi \colon R_{-i, i} \to \mathcal D_i
\) for all \(i \neq 0\) such that
\(
    \phi(\rho(u)) = u \dotplus u \cdot (-1)
\) and
\(
    \phi(ab) = a * b
\). In particular,
\(
    (R_{0i}, R_{-i, i}, \Delta_i)
\) are odd form groups for \(i \neq 0\).

We construct
\(
    \Delta^0_0
\) as an odd form parameter to the hermitian group
\(
    (R_{00}, R_{00})
\) by generators and relations (using lemmas \ref{ofg-free} and \ref{ofg-fact}). The generators are
\(
    u \cdot a
\) for
\(
    u \in \Delta^0_i
\),
\(
    a \in R_{i0}
\), \(i \neq 0\) with
\(
    \pi(u \cdot a) = \pi(u) a
\) and
\(
    \rho(u \cdot a) = \inv a \rho(u) a
\). The relations are
\begin{itemize}

    \item
    \(
        (u, a) \mapsto u \cdot a
    \) is a component of the sesquiquadratic map
    \(
        (R_{0i}, R_{-i, i}, \Delta^0_i) \times R_{i0}
        \to (R_{00}, R_{00}, \Delta^0_0)
    \) for \(i \neq 0\);

    \item
    \(
        (u \cdot a) \cdot b = u \cdot ab
    \) for
    \(
        u \in \Delta^0_i
    \),
    \(
        a \in R_{ij}
    \),
    \(
        b \in R_{j0}
    \), and
    \(
        0 \neq |i| \neq |j| \neq 0
    \);

    \item
    \(
        (u \cdot k) \cdot a = u \cdot ka
    \) for
    \(
        u \in \Delta^0_i
    \),
    \(
        a \in R_{i0}
    \), \(k \in K\), \(i \neq 0\);

    \item if
    \(
        \sum_{1 \leq s \leq N}^\cdot v_s \cdot b_s
        \dotplus \phi(a)
        = \dot 0
    \), then
    \(
        \sum_{1 \leq s \leq N}^\cdot kv_s \cdot b_s
        \dotplus \phi(ka)
        = \dot 0
    \) for all \(k \in K\), where
    \(
        v_s \in \mathcal D_{i_s}
    \),
    \(
        b_s \in R_{i_s 0}
    \), \(i_s \neq 0\).

\end{itemize}

Clearly,
\(
    \Delta^0_0
\) is a \(2\)-step nilpotent \(K\)-module with the nilpotent filtration
\(
    \mathcal D_0 = \langle
        \phi(R_{00});
        \mathcal D_i \cdot R_{i0}, i \neq 0
    \rangle
\). Namely,
\begin{itemize}

    \item
    \(
        k(v \cdot a) = kv \cdot a
    \) for
    \(
        v \in \mathcal D_i
    \),
    \(
        a \in R_{i0}
    \), \(i \neq 0\);

    \item
    \(
        (u \cdot a) \cdot k = u \cdot ka
    \) for
    \(
        u \in \Delta^0_i
    \),
    \(
        a \in R_{i0}
    \), \(i \neq 0\);

    \item
    \(
        \tau(u \cdot a)
        = \phi(\inv a \rho(u) a)
    \) for
    \(
        u \in \Delta^0_i
    \),
    \(
        a \in R_{i0}
    \), \(i \neq 0\).

\end{itemize}
The sesquiquadratic map
\(
    (R_{00}, R_{00}, \Delta^0_0) \times K
    \to (R_{00}, R_{00}, \Delta^0_0)
\) exists by lemmas \ref{ofg-sq-free}, \ref{ofg-sq-fact} and is associative by lemma \ref{ofg-sq-ass}.

The next lemma is usually applied to the abelian groups mentioned before lemma \ref{mul-cons} and the odd form groups
\(
    (R_{0i}, R_{-i, i}, \Delta^0_i)
\) and
\[
    \bigl(
        R_{0j} \oplus R_{0, -j},
        R_{-j, -j}
            \oplus R_{-j, j}
            \oplus R_{j, -j}
            \oplus R_{jj},
        \Delta^0_j
            \dotoplus \Delta^0_{-j}
            \dotoplus \phi(R_{jj})
    \bigr)
\]
for \(j \neq 0\), where the odd form parameter
\(
    \Delta^0_j
    \dotoplus \Delta^0_{-j}
    \dotoplus \phi(R_{jj})
\) is freely generated by subgroups
\(
    \Delta^0_{\pm j}
\), see lemmas \ref{ofg-free} and \ref{ofg-fact}. Sesquiquadratic maps between such odd form groups are easily constructed by ``usual'' sesquiquadratic maps
\(
    (R_{0i}, R_{-i, i}, \Delta^0_i) \times R_{ij}
    \to (R_{0j}, R_{-j, j}, \Delta^0_j)
\) using lemmas \ref{ofg-sq-free} and \ref{ofg-sq-fact}, the associativity laws for them follow from lemma \ref{ofg-sq-ass}.

\begin{lemma} \label{act-cons}
    Let \(A_{ij}\) be abelian groups for
    \(
        1 \leq i < j \leq 4
    \),
    \(
        Q_i = (M_i, H_i, \Delta_i)
    \) be odd form groups for
    \(
        1 \leq i \leq 4
    \),
    \(
        \mu_{ijk}
        \colon A_{ij} \times A_{jk}
        \to A_{ik}
    \) be biadditive maps for
    \(
        1 \leq i < j \leq 4
    \),
    \(
        (\alpha_{ij}, \beta_{ij}, \gamma_{ij})
        \colon Q_i \times A_{ij}
        \to Q_j
    \) be sesquiquadratic maps for
    \(
        1 \leq i < j \leq 4
    \) unless
    \(
        (i, j) = (2, 4)
    \), and
    \(
        (\alpha_{24}, \beta_{24})
        \colon (M_2, H_2) \times A_{24}
        \to (M_4, H_4)
    \) be a sesquiquadratic map in the sense of hermitian groups. Suppose that the odd form parameter
    \(
        \Delta^0_2
    \) is generated by the image of
    \(
        \gamma_{12}
    \),
    \(
        A_{24} = \langle
            \mu_{234}(A_{23}, A_{34})
        \rangle
    \), and all possible associative laws (including
    \(
        \mu_{134}(\mu_{123}(a, b), c)
        = \mu_{124}(a, \mu_{234}(b, c))
    \)) hold. Then
    \(
        (\alpha_{24}, \beta_{24})
    \) may be completed to a unique sesquiquadratic map
    \(
        (\alpha_{24}, \beta_{24}, \gamma_{24})
        \colon Q_2 \times A_{24}
        \to Q_4
    \) satisfying the remaining associative laws
    \(
        \gamma_{24}(\gamma_{12}(u, a), b)
        = \gamma_{14}(u, \mu_{124}(a, b))
    \) and
    \(
        \gamma_{34}(\gamma_{23}(u, a), b)
        = \gamma_{24}(u, \mu_{234}(a, b))
    \).
\end{lemma}
\begin{proof}
    Take
    \(
        u \in \Delta_2
    \) and
    \(
        a \in A_{24}
    \). Choose decompositions
    \(
        u
        = \sum_{1 \leq s \leq N}^\cdot
            \gamma_{12}(w_s, x_s)
        \dotplus \phi(h)
    \) and
    \(
        a = \sum_{t = 1}^M \mu_{234}(y_t, z_t)
    \). Writing
    \(
        \mu_{ijk}(p, q)
    \) and
    \(
        \alpha_{ij}(p, q)
    \) as \(pq\),
    \(
        \beta_{ij}(p, q, r)
    \) as
    \(
        \inv pqr
    \), and
    \(
        \gamma_{ij}(p, q)
    \) as
    \(
        p \cdot q
    \), we have
    \begin{align*}
        \sum_s^\cdot &w_s \cdot x_s a
        \dotplus \phi(\inv aha)
        = \sum_s^\cdot
            w_s \cdot x_s \bigl(
                \sum_t y_t z_t
            \bigr)
        \dotplus \phi\bigl(
            \sum_{t, t'}
                \inv{z_{t'}} \inv{y_{t'}} h y_t z_t
        \bigr) \\
        &= \sum_{s, t}^\cdot w_s \cdot x_s y_t z_t
        \dotplus \phi\bigl(
            \sum_{t, t'}
                \inv{z_{t'}} \inv{y_{t'}} h y_t z_t
            + \sum_s \sum_{t < t'}
                \inv{z_{t'}} \inv{y_{t'}} \inv{x_s}
                \rho(w_s) x_s y_t z_t
        \bigr) \\
        &= \sum_{t, s}^\cdot
            (w_s \cdot x_s y_t) \cdot z_t
        \dotplus \phi\bigl(
            \sum_{t, t'}
                \inv{z_{t'}} \inv{y_{t'}} h y_t z_t
            + \sum_{t < t'}
                \inv{z_{t'}} \inv{y_{t'}} \rho(
                    u \dotminus \phi(h)
                ) y_t z_t
        \bigr) \\
        &= \sum_{t, s}^\cdot (
            (w_s \cdot x_s) \cdot y_t
        ) \cdot z_t
        \dotplus \phi\bigl(
            \sum_t \inv{z_t} \inv{y_t} h y_t z_t
            + \sum_{t < t'}
                \inv{z_{t'}} \inv{y_{t'}}
                \rho(u) y_t z_t
        \bigr) \\
        &= \sum_t^\cdot
            (u \cdot y_t) \cdot z_t
        \dotplus \phi\bigl(
            \sum_{t < t'}
                \inv{z_{t'}} \inv{y_{t'}}
                \rho(u) y_t z_t
        \bigr).
    \end{align*}
    The common value
    \(
        \gamma_{24}(u, a)
    \) of this expressions is independent on the decompositions. Clearly,
    \(
        \gamma_{24}
    \) satisfies all required identities and is unique.
\end{proof}

\begin{lemma} \label{pres-form}
    The odd form parameter
    \(
        \Delta^0_0
    \) is generated by the formal expressions
    \(
        u \cdot a
    \) for
    \(
        u \in \Delta^0_i
    \),
    \(
        a \in R_{i0}
    \),
    \(
        |i| \in \{ 1, 2 \}
    \). The only relations between these generators are the relations from the definition of
    \(
        \Delta^0_0
    \) not involving other generators.
\end{lemma}
\begin{proof}
    Let
    \(
        \Delta^{0 \prime}_0
    \) be the odd form parameter with presentation from the statement. By lemma \ref{act-cons} there are unique maps
    \(
        ({-}) \cdot ({=})
        \colon \Delta^0_i \times R_{i0}
        \to \Delta^{0 \prime}_0
    \) for \(|i| \geq 3\) such that
    \(
        (R_{0i}, R_{-i, i}, \Delta^0_i) \times R_{i0}
        \to (R_{00}, R_{00}, \Delta^{0 \prime}_0)
    \) are sesquiquadratic maps and
    \(
        (u \cdot a) \cdot b = u \cdot ab
    \) for either
    \(
        u \in \Delta^0_{\pm 1}
    \),
    \(
        a \in R_{\pm 1, i}
    \),
    \(
        b \in R_{i0}
    \) or
    \(
        u \in \Delta^0_i
    \),
    \(
        a \in R_{i, \pm 2}
    \),
    \(
        b \in R_{\pm 2, 0}
    \).

    Let us check that
    \(
        (u \cdot a) \cdot b = u \cdot ab
    \) in
    \(
        \Delta^{0 \prime}_0
    \) for
    \(
        u \in \Delta^0_i
    \),
    \(
        a \in R_{ij}
    \),
    \(
        b \in R_{j0}
    \), and distinct non-zero \(|i|\), \(|j|\). This identity may be checked on generators by lemma \ref{ofg-sq-ass}:
    \begin{itemize}

        \item If
        \(
            |i|, |j| \in \{ 1, 2 \}
        \), or \(|i| = 1\) and \(|j| \geq 3\), or \(|i| \geq 3\) and \(|j| = 2\), then this follows from the definitions.

        \item If \(|i| = 2\) and \(|j| \geq 3\), then we have
        \[
            (u \cdot xy) \cdot b
            = ((u \cdot x) \cdot y) \cdot b
            = (u \cdot x) \cdot yb
            = u \cdot xyb
        \]
        for
        \(
            x \in R_{i, \pm 1}
        \),
        \(
            y \in R_{\pm 1, j}
        \).

        \item If \(|i| \geq 3\) and \(|j| \neq 2\), then we have
        \[
            (u \cdot a) \cdot xy
            = ((u \cdot a) \cdot x) \cdot y
            = (u \cdot ax) \cdot y
            = u \cdot axy
        \]
        for
        \(
            x \in R_{j, \pm 2}
        \),
        \(
            y \in R_{\pm 2, 0}
        \).

    \end{itemize}

    The identity
    \(
        (u \cdot k) \cdot a = u \cdot ka
    \) for
    \(
        u \in \Delta^0_i
    \),
    \(
        a \in R_{i0}
    \), \(k \in K\),
    \(
        |i| \geq 3
    \) follows from the same lemma and
    \[
        (u \cdot k) \cdot xy
        = ((u \cdot k) \cdot x) \cdot y
        = (u \cdot kx) \cdot y
        = u \cdot kxy
    \]
    for
    \(
        x \in R_{i, \pm 1}
    \),
    \(
        y \in R_{\pm 1, 0}
    \).

    Finally, suppose that
    \(
        \sum^\cdot_{1 \leq s \leq N} v_s \cdot b_s
        \dotplus \phi(a)
        = \dot 0
    \) for some
    \(
        v_s \in \mathcal D_{i_s}
    \),
    \(
        b_s \in R_{i_s 0}
    \), \(i_s \neq 0\). For each \(s\) with
    \(
        |i_s| \geq 3
    \) choose a decomposition
    \(
        b_s = \sum_{1 \leq t \leq M_s} c_{st} d_{st}
    \) for some
    \(
        c_{st} \in R_{i_s j_{st}}
    \),
    \(
        d_{st} \in R_{j_{st} 0}
    \),
    \(
        |j_{st}| = 1
    \). We have
    \[
        \sum^\cdot_{|i_s| \in \{ 1, 2 \}}
            v_s \cdot b_s
        \dotplus \sum^\cdot_{|i_s| \geq 3}
            \sum^\cdot_t
                (v_s \cdot c_{st}) \cdot d_{st}
        \dotplus
        \phi\bigl(
            \sum_{\substack{|i_s| \geq 3\\ t < t'} }
                \inv{d_{st'}} \inv{c_{st'}}
                \rho(v_s) c_{st} d_{st}
            + a
        \bigr)
        = \dot 0,
    \]
    so
    \[
        \sum^\cdot_s kv_s \cdot b_s
        \dotplus \phi(ka)
        = \sum^\cdot_{|i_s| \in \{ 1, 2 \}}
            kv_s \cdot b_s
        \dotplus \sum^\cdot_{|i_s| \geq 3}
            \sum^\cdot_t
                k(v_s \cdot c_{st}) \cdot d_{st}
        \dotplus \phantom x
        \phi\bigl(
            \sum_{\substack{|i_s| \geq 3\\ t < t'} }
                k \inv{d_{st'}} \inv{c_{st'}}
                \rho(v_s) c_{st} d_{st}
            + ka
        \bigr)
        = \dot 0.
    \qedhere\]
\end{proof}

\begin{lemma} \label{form-struct}
    There are unique maps
    \(
        ({-}) \cdot ({=})
        \colon \Delta^0_i \times R_{ij}
        \to \Delta^0_j
    \) for \(i = 0\) and for \(|i| = |j|\) making
    \(
        (R_{0i}, R_{-i, i}, \Delta^0_i) \times R_{ij}
        \to (R_{0j}, R_{-j, j}, \Delta^0_j)
    \) sesquiquadratic maps satisfying associativity laws for all \(i\) and \(j\).
\end{lemma}
\begin{proof}
    Let us denote by
    \(
        \mathrm A[i, j, k]
    \) the associative law
    \(
        (u \cdot a) \cdot b = u \cdot ab
    \) for
    \(
        u \in \Delta^0_i
    \),
    \(
        a \in R_{ij}
    \),
    \(
        b \in R_{jk}
    \). It holds for distinct \(|i|\), \(|j|\), \(|k|\) such that \(|i|\) and \(|j|\) are non-zero. In general such law may be checked on generators by lemma \ref{ofg-sq-ass}.

    We construct the sesquiquadratic maps
    \(
        (R_{00}, R_{00}, \Delta^0_0) \times R_{0i}
        \to (R_{0i}, R_{-i, i}, \Delta^0_i)
    \) for \(i \neq 0\) satisfying
    \(
        \mathrm A[k, 0, i]
    \) for
    \(
        0 \neq |k| \neq |i|
    \). Without loss of generality, \(|i| = 3\). By lemmas \ref{ofg-sq-free}, \ref{ofg-sq-fact}, \ref{pres-form} there is unique such sesquiquadratic map satisfying
    \(
        \mathrm A[k, 0, i]
    \) for
    \(
        |k| \in \{ 1, 2 \}
    \). If \(|k| > 3\), then we use lemma \ref{ofg-sq-ass} and
    \[
        (u \cdot xy) \cdot a
        = ((u \cdot x) \cdot y) \cdot a
        = (u \cdot x) \cdot ya
        = u \cdot xya
    \]
    for
    \(
        u \in \Delta^0_k
    \),
    \(
        x \in R_{k, \pm 1}
    \),
    \(
        y \in R_{\pm 1, 0}
    \),
    \(
        a \in R_{0i}
    \).

    Now let us check
    \(
        \mathrm A[i, j, k]
    \) if \(|i|\), \(|j|\), \(|k|\) are distinct using lemma \ref{ofg-sq-ass}. The only remaining case is \(i = 0\), it follows from
    \[
        ((u \cdot a) \cdot b) \cdot c
        = (u \cdot ab) \cdot c
        = u \cdot abc
        = (u \cdot a) \cdot bc
    \]
    for
    \(
        u \in \Delta^0_l
    \),
    \(
        a \in R_{l0}
    \),
    \(
        b \in R_{0j}
    \),
    \(
        c \in R_{jk}
    \), where
    \(
        |l| \notin \{ 0, |j|, |k| \}
    \). Here we use the corollary
    \[
        \Delta^0_0 = \langle
            \phi(R_{00}),
            \Delta^0_l \cdot R_{l0},
            \Delta^0_{-l} \cdot R_{-l, 0}
        \rangle
    \]
    of lemma \ref{pres-form} and
    \(
        \mathrm A[1, 2, 0]
    \).

    We are ready to construct the sesquiquadratic maps
    \(
        (R_{0i}, R_{-i, i}, \Delta^0_i) \times R_{ij}
        \to (R_{0j}, R_{-j, j}, \Delta^0_j)
    \) for \(|i| = |j|\). Without loss of generality,
    \(
        |i| \in \{ 0, 1 \}
    \). By lemma \ref{act-cons} there is unique such map satisfying
    \(
        \mathrm A[\pm 2, i, j]
    \) and
    \(
        \mathrm A[i, \pm 3, j]
    \). The identity
    \(
        \mathrm A[k, i, j]
    \) for
    \(
        |k| \notin \{ |i|, 2 \}
    \) follows from
    \[
        (u \cdot ab) \cdot c
        = ((u \cdot a) \cdot b) \cdot c
        = (u \cdot a) \cdot bc
        = u \cdot abc
    \]
    for
    \(
        u \in \Delta^0_k
    \),
    \(
        a \in R_{k, \pm 2}
    \),
    \(
        b \in R_{\pm 2, i}
    \),
    \(
        c \in R_{ij}
    \). Similarly, the identity
    \(
        \mathrm A[i, k, j]
    \) for
    \(
        |k| \notin \{ |i|, 3 \}
    \) follows from
    \[
        (u \cdot ab) \cdot c
        = ((u \cdot a) \cdot b) \cdot c
        = (u \cdot a) \cdot bc
        = u \cdot abc
    \]
    for
    \(
        u \in \Delta^0_i
    \),
    \(
        a \in R_{i, \pm 3}
    \),
    \(
        b \in R_{\pm 3, k}
    \),
    \(
        c \in R_{kj}
    \).

    The remaining associative law
    \(
        \mathrm A[i, j, k]
    \) for \(|i| = |j|\) follows from
    \[
        (u \cdot ab) \cdot c
        = ((u \cdot a) \cdot b) \cdot c
        = (u \cdot a) \cdot bc
        = u \cdot abc
    \]
    for
    \(
        u \in \Delta^0_i
    \),
    \(
        a \in R_{im}
    \),
    \(
        b \in R_{mj}
    \),
    \(
        c \in R_{jk}
    \), and some
    \(
        m \notin \{ 0, |i|, |k| \}
    \).
\end{proof}

\begin{lemma} \label{peirce-ofr}
Let \(R_{ij}\) be \(K\)-modules for
\(
    -\ell \leq i, j \leq \ell
\) with multiplication maps
\(
    R_{ij} \times R_{jk} \to R_{ik}
\) and isomorphisms
\(
    \inv{(-)} \colon R_{ij} \to R_{-j, -i}
\) making
\(
    R = \bigoplus_{i, j = -\ell}^\ell R_{ij}
\) an associative \(K\)-algebra with involution. Let also
\(
    (R_{0i}, R_{-i, i}, \Delta^0_i)
\) be odd form groups for
\(
    -\ell \leq i \leq \ell
\) (with
\(
    \langle a, b \rangle = \inv ab
\)) and
\begin{align*}
    \bigl(
        (a, b) \mapsto ab,
        (a, b, c) \mapsto \inv abc,
        (u, a) \mapsto u \cdot a
    \bigr)
    &\colon (R_{0i}, R_{-i, i}, \Delta^0_i)
    \times R_{ij}
    \to (R_{0j}, R_{-j, j}, \Delta^0_j);
    \\
    \bigl(
        (a, k) \mapsto ak,
        (k, a, k') \mapsto akk',
        (u, k) \mapsto u \cdot k
    \bigr)
    &\colon (R_{0i}, R_{-i, i}, \Delta^0_i) \times K
    \to (R_{0i}, R_{-i, i}, \Delta^0_i)
\end{align*}
be sesquiquadratic maps satisfying all associativity laws. Suppose in addition that
\(
    \Delta^0_i
\) are \(2\)-step nilpotent \(K\)-modules with the nilpotent filtrations
\(
    \phi(R) \leq \mathcal D_i \leq \Ker(\pi)
\),
\(
    \mathcal D_i \cdot R_{ij} \subseteq \mathcal D_j
\),
\(
    \phi(ka) = k \phi(a)
\),
\(
    \rho(kv) = k \rho(v)
\),
\(
    kv \cdot a = k(v \cdot a)
\) for
\(
    v \in \mathcal D_i
\). Then
\(
    (R, \Delta, \mathcal D)
\) is a graded odd form \(K\)-algebra, where
\begin{align*}
    \Delta
    &= \bigoplus_i^\cdot \Delta^0_i
    \dotoplus \bigoplus_{i \neq 0;\, j}^\cdot
        q_i \cdot R_{ij}
    \dotoplus \bigoplus_{i + j > 0}^\cdot
        \phi(R_{ij});
    &
    \mathcal D
    &= \bigoplus_i^\cdot \mathcal D_i
    \dotoplus \bigoplus_{i + j > 0}^\cdot
        \phi(R_{ij}).
\end{align*}
\end{lemma}
\begin{proof}
By lemmas \ref{ofg-free} and \ref{ofg-fact},
\(
    (R, R, \Delta)
\) is an odd form group. The sesquiquadratic map
\(
    (R, R, \Delta) \times (R \rtimes K)
    \to (R, R, \Delta)
\) exists by lemmas \ref{ofg-sq-free} and \ref{ofg-sq-fact}, it satisfies the associativity law by lemma \ref{ofg-sq-ass}. Clearly, \(\mathcal D\) is an augmentation of
\(
    (R, \Delta)
\).
\end{proof}

\begin{theorem} \label{coord-cryst}
    Let \(G\) be a group with a family of subgroups
    \(
        G_\alpha \leq G
    \) indexed by roots of a root system of type
    \(
        \mathsf{BC}_\ell
    \) for
    \(
        \ell \geq 3
    \). Suppose that \ref{comm-rel}--\ref{homo-comm} hold for a commutative unital ring \(K\), as well as the associativity conditions \ref{ass-law-l}--\ref{ass-law-bal} (e.g. if
    \(
        \ell \geq 4
    \)). Then there is a graded odd form \(K\)-algebra
    \(
        (R, \Delta)
    \) satisfying the idempotency conditions and a surjective homomorphism
    \(
        \stunit(R, \Delta) \to G
    \) inducing isomorphisms between root subgroups. Moreover, any partial graded odd form \(K\)-algebra satisfying \ref{short-idem}, \ref{long-idem}, and the associativity conditions arises from such a group \(G\).
\end{theorem}
\begin{proof}
    We are going to apply lemmas \ref{ring-struct}, \ref{form-struct}, \ref{peirce-ofr}. The only non-trivial remaining condition is the associativity law
    \(
        (u \cdot a) \cdot k
        = (u \cdot k) \cdot a
        = u \cdot ka
    \) for
    \(
        u \in \Delta^0_i
    \), \(k \in K\),
    \(
        a \in R_{ij}
    \). If
    \(
        |i| \notin \{ 0, |j| \}
    \), then such identity is already known. Otherwise this follows from
    \begin{align*}
        ((u \cdot a) \cdot b) \cdot k
        = (u \cdot ab) \cdot k
        &= u \cdot kab;
        \\
        ((u \cdot a) \cdot k) \cdot b
        = (u \cdot ka) \cdot b
        &= u \cdot kab;
        \\
        (u \cdot a) \cdot kb
        &= u \cdot kab
    \end{align*}
    for
    \(
        u \in \Delta^0_k
    \), \(k \in K\),
    \(
        a \in R_{ki}
    \),
    \(
        b \in R_{ij}
    \),
    \(
        |k| \notin \{ 0, |i|, |j| \}
    \).
\end{proof}

\section{Existence of action of \(\mathbb Z\)}

In this section we use the identities \ref{mul-cen}--\ref{dot-ass}, the conditions \ref{long-idem}, \ref{short-idem}, and lemma \ref{idem-par-ofr} without explicit references.

It turns out that not every partial graded odd form ring (constructed by a group with subgroups satisfying \ref{comm-rel}--\ref{short-gen}) is a partial graded odd form \(\mathbb Z\)-algebra.

\begin{example}
    Let \(A\) be the non-unital commutative ring generated by \(x_i\) for \(i \geq 0\) with the relations
    \(
        x_0 x_i = 0
    \),
    \(
        x_{i + 1}^2 = x_i
    \). It is an idempotent ring (i.e. \(A = \langle AA \rangle\)) and \(x_0 A = 0\). Since \(A\) as an abelian group is free with the basis
    \(
        \{
            x_0^\alpha
        \mid
            \alpha \in \mathbb Z[1 / 2],
            0 < \alpha \leq 1
        \}
    \), \(nx_0 \neq 0\) for all
    \(
        n \in \mathbb Z \setminus \{ 0 \}
    \). Let
    \(
        (R, \Delta)
    \) be the symplectic odd form algebra over \(A\) of rank
    \(
        \ell \geq 4
    \) constructed in \cite[\S 2.3]{thesis} and \cite[\S 5]{classic-ofa}, i.e. it is the graded odd form
    \(
        (A \rtimes \mathbb Z)
    \)-algebra with
    \begin{align*}
        R_{ij} &= A e_{ij} \text{ for } ij \neq 0;
        &
        e_{ij} e_{jk} &= e_{ik};
        \\
        \Delta^0_i &= \mathcal D_i = A v_i
        \text{ for } i \neq 0;
        &
        \inv{e_{ij}} &= \eps_i \eps_j e_{-j, -i};
        \\
        R_{i0} &= \Delta^0_0 = \dot 0;
        &
        \phi(e_{-i, i}) &= 2 v_i;
        \\
        v_i \cdot e_{ij} &= \eps_i \eps_j v_j;
        &
        \rho(v_i) &= e_{-i, i};
    \end{align*}
    where
    \(
        \eps_i = 1
    \) for \(i > 0\) and
    \(
        \eps_i = -1
    \) for \(i < 0\). We consider the Peirce decomposition of
    \(
        (R, \Delta)
    \) of rank \(\ell - 1\) obtained from the standard one by ``eliminating'' the root \(\mathrm e_1\), i.e. let
    \begin{align*}
        R_{ij} &= A e_{ij};
        &
        R_{00}
        &= Ae_{11}
        \oplus Ae_{-1, 1}
        \oplus Ae_{1, -1}
        \oplus Ae_{-1, -1};
        \\
        R_{i0} &= Ae_{i1} \oplus Ae_{i, -1};
        &
        \Delta^i_0
        &= q_i \cdot Ae_{i1}
        \dotoplus q_i \cdot Ae_{i, -1};
        \\
        \Delta^i_j &= q_i \cdot Ae_{ij};
        &
        \Delta^0_i
        &= Av_i
        \dotoplus q_1 \cdot Ae_{1i}
        \dotoplus q_{-1} \cdot Ae_{-1, i};
        \\
        \Delta^0_0
        &= q_1 \cdot Ae_{11}
        \dotoplus q_1 \cdot Ae_{1, -1}
        \dotoplus q_{-1} \cdot Ae_{-1, 1}
        \dotoplus q_{-1} \cdot Ae_{-1, -1}
        \dotoplus \phi(Ae_{11})
        \dotoplus Av_1
        \dotoplus Av_{-1}\mkern-500mu
    \end{align*}
    for
    \(
        2 \leq |i|, |j| \leq \ell
    \).

    Let
    \[
        \Gamma = \{
            nx_0 v_2
            \dotplus q_1 \cdot nx_0 e_{12}
            \dotplus q_{-1} \cdot nx_0 e_{-1, 2}
        \mid
            n \in \mathbb Z
        \} \leq \Delta^0_2,
    \]
    it is a central subgroup. Then the commutator maps between the root subgroups of the Steinberg group
    \(
        \stunit(R, \Delta)
    \) factor through
    \(
        \Delta^0_2 / \Gamma
    \) unlike the operation
    \(
        (-) \cdot (-1)
    \) since
    \[
        (
            x_0 v_2
            \dotplus q_1 \cdot x_0 e_{12}
            \dotplus q_{-1} \cdot x_0 e_{-1, 2}
        ) \cdot (-1)
        = x_0 v_2
        \dotminus q_1 \cdot x_0 e_{12}
        \dotminus q_{-1} \cdot x_0 e_{-1, 2}
        \notin \Gamma.
    \]
    It follows that
    \(
        \stunit(R, \Delta) / X_2(\Gamma)
    \) satisfies \ref{comm-rel}--\ref{short-gen} (the condition \ref{non-deg} follows from \cite[lemma 11]{thesis}), but does not satisfy \ref{bc-comm-rel}--\ref{homo-comm} for any choice of
    \(
        \mathcal D_i
    \) and
    \(
        K = \mathbb Z
    \).
\end{example}

\begin{lemma} \label{idem-cond}
    Let
    \(
        (R_{ij}, \Delta^0_i)_{ij}
    \) be a partial graded odd form ring satisfying \ref{long-idem} and \ref{short-idem} such that for every \(i \neq 0\) there is a group homomorphism
    \(
        ({-}) \cdot (-1)
        \colon \Delta^0_i \to \Delta^0_i
    \) satisfying
    \(
        (a * b) \cdot (-1) = a * b
    \) and
    \(
        (u \cdot b) \cdot (-1) = u \cdot (-b)
    \) for
    \(
        a \in R_{-i, j}
    \),
    \(
        b \in R_{ji}
    \),
    \(
        u \in \Delta^0_j
    \), and distinct non-zero \(|i|\), \(|j|\). Then
    \(
        (R_{ij}, \Delta^0_i)_{ij}
    \) has a structure of a partial graded odd form \(\mathbb Z\)-algebra with the smallest augmentation
    \(
        \mathcal D^{\mathrm{min}}_i
    \) from lemma \ref{idem-par-ofr} and the largest one
    \begin{align*}
        \mathcal D^{\mathrm{max}}_i &= \{
            u \in \Delta^0_i
        \mid
            u = u \cdot (-1),
            u \circ \Delta^0_j = 0
            \text{ for all }
            j \notin \{ 0, i, -i \}
        \} \\
        &= \{
            u \in \Delta^0_i
        \mid
            u = u \cdot (-1),
            u \circ \Delta^0_{j_*}
            = u \circ \Delta^0_{-j_*}
            = 0
        \}
    \end{align*}
    for any fixed
    \(
        j_* \notin \{ 0, i, -i \}
    \). The action of the multiplicative monoid
    \(
        \mathbb Z^\bullet
    \) on
    \(
        \Delta^0_i
    \) is uniquely determined.
\end{lemma}
\begin{proof}
    Clearly, such homomorphisms
    \(
        ({-}) \cdot (-1)
    \) are unique. Also,
    \(
        [u \cdot (-1), v] = \dotminus [u, v]
    \) for all
    \(
        u, v \in \Delta^0_i
    \). Let
    \(
        u \cdot n
        = {\binom{n + 1}2}^\cdot u
        \dotplus {\binom n 2}^\cdot u \cdot (-1)
    \) for
    \(
        u \in \Delta^0_i
    \) and
    \(
        n \in \mathbb Z
    \), where \(n^\cdot u\) denotes the \(n\)-th multiple of \(u\), this is the only possible choice. This formula gives the required structure for
    \(
        \mathcal D_i = \mathcal D_i^{\mathrm{min}}
    \) or
    \(
        \mathcal D_i = \mathcal D_i^{\mathrm{max}}
    \).
\end{proof}

There are two natural classes of partial graded odd form rings that always satisfy the statement of lemma \ref{idem-cond}, we call them centerless and centrally closed ones. Let
\(
    (R_{ij}, \Delta^0_i)_{ij}
\) be a partial graded odd form ring satisfying \ref{long-idem}, \ref{short-idem}, and the associativity conditions. Its \textit{center} is
\(
    (I_{ij}, \Gamma^0_i)_{ij}
\), where \(I_{ij}\) consists of
\(
    a \in R_{ij}
\) such that \(
    a R_{jk} = a R_{j, -k} = 0
\) for some (or for all)
\(
    |k| \notin \{ 0, |i|, |j| \}
\) and
\(
    R_{ki} a = R_{-k, i} a = 0
\) for some (or for all)
\(
    |k| \notin \{ 0, |i|, |j| \}
\);
\(
    \Gamma^0_i
\) consists of
\(
    u \in \Delta^0_i
\) such that
\(
    u \cdot R_{ij} = u \cdot R_{i, -j} = \dot 0
\) and
\(
    u \circ \Delta^0_j
    = u \circ \Delta^0_{-j}
    = u \triangleleft R_{ij}
    = u \triangleleft R_{i, -j}
    = 0
\) for some (or for all)
\(
    |j| \notin \{ 0, |i| \}
\). If
\(
    a \in I_{ij}
\), then
\(
    \inv a \in I_{-j, -i}
\),
\(
    \Delta^0_i \triangleleft a = 0
\),
\(
    a * R_{j, -i} = \dot 0
\),
\(
    \Delta^0_i \cdot a = \dot 0
\).

A surjective homomorphism
\(
    f = (f_{ij}, f_i^0)_{ij}
    \colon (R_{ij}, \Delta^0_i)_{ij}
    \to (S_{ij}, \Theta^0_i)_{ij}
\) of partial graded odd form rings (in the sense of universal algebra) satisfying \ref{long-idem}, \ref{short-idem}, and the associativity conditions is a \textit{central extension} if its kernel
\(
    \Ker(f) = (\Ker(f_{ij}), \Ker(f_i^0))_{ij}
\) is contained in the center of
\(
    (R_{ij}, \Delta^0_i)_{ij}
\).

We say that a partial graded odd form ring satisfying \ref{long-idem}, \ref{short-idem}, and the associativity conditions is \textit{centrally closed} if is does not have non-trivial central extensions satisfying these conditions and \textit{centerless} if it is not a non-trivial central extension.

\begin{lemma} \label{cext-comp}
    A composition of central extensions of partial graded odd form rings satisfying \ref{long-idem}, \ref{short-idem}, and \ref{ass-law-l}--\ref{ass-law-bal} is also a central extension.
\end{lemma}
\begin{proof}
    Let
    \(
        (f_{ij}, f^0_i)_{ij}
        \colon (R_{ij}, \Delta^0_i)_{ij}
        \to (S_{ij}, \Theta^0_i)_{ij}
    \) be a central extension of graded odd form rings satisfying the required conditions. In the following all indices are non-zero and with distinct absolute values unless otherwise stated. If \(f^0_i(u)\) is central for some
    \(
        u \in \Delta^0_i
    \), then clearly \(u\) is also central. If \(f_{ij}(a)\) is central for some
    \(
        a \in R_{ij}
    \), then \(a\) is central since
    \begin{align*}
        a R_{jn}
        &\subseteq \bigl\langle
            a ((R_{jk} (R_{kl} R_{lm})) R_{mn})
            \mid
            |i| = |m|, |j| = |l|, |k| = |n|
        \bigr\rangle \\
        &= \bigl\langle
            ((a R_{jk}) R_{kl}) (R_{lm} R_{mn})
            \mid
            |i| = |m|, |j| = |l|, |k| = |n|
        \bigr\rangle
        = 0
    \end{align*}
    by \ref{ass-law-l} and similarly for \(\inv a\).
\end{proof}

\begin{lemma} \label{centerless}
    Let
    \(
        (R_{ij}, \Delta^0_i)_{ij}
    \) be a partial graded odd form ring satisfying \ref{long-idem}, \ref{short-idem}, and the associativity conditions with center
    \(
        (I_{ij}, \Gamma^0_i)_{ij}
    \). Then
    \(
        (S_{ij}, \Theta^0_i)_{ij}
        = (
            R_{ij} / I_{ij},
            \Delta^0_i / \Gamma^0_i
        )_{ij}
    \) is a centerless partial graded odd form ring with a structure of a partial graded odd form \(\mathbb Z\)-algebra.
\end{lemma}
\begin{proof}
    It is easy to see that
    \(
        \Gamma^0_i \leqt \Delta^0_i
    \) and
    \(
        (S_{ij}, \Theta^0_i)_{ij}
    \) is a well-defined partial graded odd form ring. Suppose that
    \[
        v
        = \sum_{1 \leq s \leq N}^\cdot u_s \cdot a_s
        \dotplus \sum_{1 \leq t \leq M}^\cdot
            b_t * c_t
        \in \Gamma^0_i
    \]
    for some
    \(
        u_s \in \Delta^0_{j_s}
    \),
    \(
        a_s \in R_{j_s i}
    \),
    \(
        b_t \in R_{-i, k}
    \),
    \(
        c_t \in R_{k i}
    \), and
    \(
        |j_s| = k \notin \{ 0, |i| \}
    \). We claim that
    \[
        v'
        = \sum_{1 \leq s \leq N}^\cdot
            u_s \cdot (-a_s)
        \dotplus \sum_{1 \leq t \leq M}^\cdot
            b_t * c_t
    \]
    also lies in
    \(
        \Gamma^0_i
    \). Indeed,
    \(
        v' \circ w = -v \circ w = 0
    \),
    \(
        v' \triangleleft x = v \triangleleft x = 0
    \),
    \(
        v' \cdot x = v \cdot (-x) = \dot 0
    \) for
    \(
        x \in R_{il}
    \),
    \(
        w \in \Delta^0_l
    \),
    \(
        |l| \notin \{ 0, |i|, k \}
    \). The claim now follows from lemma \ref{idem-cond}.
\end{proof}

Now we are going to describe universal central extensions of partial graded odd form rings. By lemmas \ref{cext-comp}, \ref{centerless} and theorem \ref{coord-cryst}, it suffices to consider only graded odd form \(\mathbb Z\)-algebras
\(
    (R, \Delta)
\) satisfying the idempotency conditions. Let
\(
    \widetilde R_{ij}
\) be the abelian group generated by formal products
\(
    a \otimes b
\) for
\(
    a \in R_{ik}
\),
\(
    b \in R_{kj}
\) satisfying the relations
\begin{itemize}

    \item
    \(
        (a, b) \mapsto a \otimes b
    \) is biadditive;

    \item
    \(
        ab \otimes c = a \otimes bc
    \) for
    \(
        a \in R_{ik}
    \),
    \(
        b \in R_{kl}
    \),
    \(
        c \in R_{kj}
    \).

\end{itemize}
There are natural homomorphisms
\(
    f_{ij} \colon \widetilde R_{ij} \to R_{ij},\,
    a \otimes b \mapsto ab,
\) and maps
\begin{align*}
    \inv{ (-) }
    &\colon \widetilde R_{ij}
    \to \widetilde R_{-j, -i},
    a \otimes b \mapsto \inv{\,b\,} \otimes \inv a,
    \\
    ({-}) ({=})
    &\colon \widetilde R_{ij} \times \widetilde R_{jk}
    \to \widetilde R_{ik},
    (a \otimes b, c \otimes d)
    \mapsto ab \otimes cd
\end{align*}
such that
\(
    \widetilde R = \bigoplus_{ij} \widetilde R_{ij}
\) is an associative ring with involution and
\(
    \bigoplus_{ij} f_{ij}
\) is a ring homomorphism. Using lemmas \ref{ofg-free} and \ref{ofg-fact} we construct the odd form parameter
\(
    \widetilde \Delta^0_i
\) for the hermitian group
\(
    (\widetilde R_{0i}, \widetilde R_{-i, i})
\) generated by formal products
\(
    u \boxtimes a
\) for
\(
    u \in \Delta^0_j
\),
\(
    a \in R_{ji}
\) satisfying the relations
\begin{itemize}

    \item
    \(
        (u, a) \mapsto u \boxtimes a
    \) is a component of a sesquiquadratic map (in particular,
    \(
        \pi(u \boxtimes a) = \pi(u) \otimes a
    \) and
    \(
        \rho(u \boxtimes a) = \inv a \otimes \rho(u) a
    \));

    \item
    \(
        (u \cdot a) \boxtimes b = u \boxtimes ab
    \) for
    \(
        u \in \Delta^0_k
    \),
    \(
        a \in R_{kj}
    \),
    \(
        b \in R_{ji}
    \).

\end{itemize}

The components of sesquiquadratic maps
\begin{align*}
    ({-}) \cdot ({=})
    &\colon \widetilde \Delta^0_i
    \times \widetilde R_{ij}
    \to \widetilde \Delta^0_j,\,
    (u \boxtimes a, b \otimes c)
    \mapsto u \boxtimes abc,
    \\
    ({-}) \cdot ({=})
    &\colon \widetilde \Delta^0_i \times \mathbb Z
    \to \widetilde \Delta^0_i,\
    (u \boxtimes a, n) \mapsto u \boxtimes an
\end{align*}
are well-defined by lemmas \ref{ofg-sq-free}, \ref{ofg-sq-fact} and make
\(
    (\widetilde R, \widetilde \Delta)
\) a graded odd form \(\mathbb Z\)-algebra with
\(
    \widetilde{ \mathcal D_i }
    = \phi(\widetilde R_{-i, i})
\) by lemmas \ref{ofg-sq-ass}, \ref{peirce-ofr}. The maps
\(
    f_i \colon \widetilde \Delta^0_i \to \Delta^0_i,
    u \boxtimes a \mapsto ua
\) are components of a homomorphism
\(
    f = (f_{ij}, f_i)_{ij}
\) of graded odd form \(\mathbb Z\)-algebras. Clearly,
\(
    ab = f_{ij}(a) \otimes f_{jk}(b)
\) and
\(
    u \cdot a = f^0_i(u) \boxtimes f_{ij}(a)
\) for all
\(
    a \in \widetilde R_{ij}
\),
\(
    b \in \widetilde R_{jk}
\),
\(
    u \in \widetilde \Delta^0_i
\).

\begin{lemma} \label{ucext-pres}
    The groups
    \(
        \widetilde R_{ij}
    \) are generated by
    \(
        a \otimes b
    \) for
    \(
        a \in R_{ik}
    \),
    \(
        b \in R_{kj}
    \),
    \(
        |k| \notin \{ 0, |i|, |j| \}
    \) with the relations
    \begin{itemize}

        \item
        \(
            (a, b) \mapsto a \otimes b
        \) are biadditive;

        \item
        \(
            ab \otimes c = a \otimes bc
        \) for
        \(
            a \in R_{ik}
        \),
        \(
            b \in R_{kl}
        \),
        \(
            c \in R_{lj}
        \), and
        \(
            |k|, |l| \notin \{ 0, |i|, |j| \}
        \).

    \end{itemize}
    The odd form parameters
    \(
        \widetilde \Delta^0_i
    \) are generated by
    \(
        u \boxtimes a
    \) for
    \(
        u \in \Delta^0_j
    \),
    \(
        a \in R_{ji}
    \),
    \(
        |j| \notin \{ 0, |i| \}
    \) with the relations
    \begin{itemize}

        \item
        \(
            (u, a) \mapsto u \boxtimes a
        \) are components of sesquiquadratic maps;

        \item
        \(
            (u \cdot a) \boxtimes b = u \boxtimes ab
        \) for
        \(
            u \in \Delta^0_k
        \),
        \(
            a \in R_{kj}
        \),
        \(
            b \in R_{ji}
        \), and
        \(
            |j|, |k| \notin \{ 0, |i| \}
        \).

    \end{itemize}
\end{lemma}
\begin{proof}
    The remaining generators of
    \(
        \widetilde R_{ij}
    \) and
    \(
        \widetilde \Delta_i
    \) may be constructed using lemmas \ref{mul-cons} and \ref{act-cons}. The remaining relations follow from
    \[
        abc \otimes d = ab \otimes cd = a \otimes bcd
    \]
    for
    \(
        a \in R_{ik}
    \),
    \(
        b \in R_{kl}
    \),
    \(
        c \in R_{lm}
    \),
    \(
        d \in R_{mj}
    \),
    \(
        |l| \notin \{ 0, |i|, |j| \}
    \);
    \[
        (u \cdot ab) \boxtimes c
        = (u \cdot a) \boxtimes bc
        = u \boxtimes abc
    \]
    for
    \(
        u \in \Delta^0_j
    \),
    \(
        a \in R_{jk}
    \),
    \(
        b \in R_{kl}
    \),
    \(
        c \in R_{li}
    \),
    \(
        |k| \notin \{ 0, |i| \}
    \); and lemma \ref{ofg-sq-ass}.
\end{proof}

\begin{lemma} \label{ucext}
    The partial graded odd form ring
    \(
        (
            \widetilde R_{ij},
            \widetilde \Delta^0_i
        )_{ij}
    \) is the universal central extension of
    \(
        (R_{ij}, \Delta^0_i)_{ij}
    \). In particular, it is centrally closed.
\end{lemma}
\begin{proof}
    Clearly,
    \(
        (
            \widetilde R_{ij},
            \widetilde \Delta^0_i
        )_{ij}
    \) is a central extension of
    \(
        (R_{ij}, \Delta^0_i)_{ij}
    \). Now let
    \(
        (g_{ij}, g^0_i)_{ij}
        \colon (S_{ij}, \Theta^0_i)_{ij}
        \to (R_{ij}, \Delta^0_i)_{ij}
    \) be any central extension of partial graded odd form rings satisfying \ref{long-idem}, \ref{short-idem}, and the associativity conditions. There is at most one homomorphism
    \(
        (h_{ij}, h^0_i)_{ij}
        \colon (
            \widetilde R_{ij},
            \widetilde \Delta^0_i
        )_{ij} \to (S_{ij}, \Theta^0_i)_{ij}
    \) making the resulting diagram commutative since
    \begin{align*}
        h_{ij}(g_{ik}(a) \otimes g_{kj}(b)) &= ab;
        &
        h^0_i(\phi(
            g_{-i, j}(a) \otimes g_{ji}(b)
        ) ) &= a * b;
        &
        h^0_i(g^0_j(u) \boxtimes g_{ji}(a))
        &= u \cdot a
    \end{align*}
    for distinct non-zero \(|i|\), \(|j|\), \(|k|\). These identities define unique group homomorphisms
    \(
        h_{ij} \colon \widetilde R_{ij} \to S_{ij}
    \) and
    \(
        h^0_i
        \colon \widetilde \Delta^0_i
        \to \Theta^0_i
    \) for
    \(
        0 \neq |i| \neq |j| \neq 0
    \) by \ref{ass-law-l} and lemmas \ref{ofg-free}, \ref{ofg-fact}, \ref{ucext-pres}. Moreover,
    \(
        (h_{ij}, h^0_i)_{ij}
    \) is a homomorphism of partial graded odd form rings.
\end{proof}

\begin{theorem} \label{coord-sphere}
    Any partial graded odd form ring with
    \(
        \ell \geq 3
    \) satisfying \ref{long-idem}, \ref{short-idem}, and the associativity conditions (e.g. if
    \(
        \ell \geq 4
    \)) arises from a group \(G\) with a family of subgroups
    \(
        G_\alpha \leq G
    \) indexed by roots of a root system of type
    \(
        \mathsf{BC}_\ell
    \) satisfying \ref{comm-rel}--\ref{short-gen}. If such \(G\) is centerless or centrally closed, then the partial graded odd form ring completes to an odd form \(\mathbb Z\)-algebra.
\end{theorem}
\begin{proof}
    Let
    \(
        (R_{ij}, \Delta^0_i)_{ij}
    \) is a partial graded odd form ring satisfying \ref{long-idem}, \ref{short-idem}, and the associativity conditions. Its universal central extension
    \(
        (
            \widetilde R_{ij},
            \widetilde \Delta^0_i
        )_{ij}
    \) completes to an odd form ring
    \(
        (R, \Delta)
    \) by lemma \ref{ucext} and theorem \ref{coord-cryst}. Let
    \(
        I_{ij} = \Ker(\widetilde R_{ij} \to R_{ij})
    \) and
    \(
        \Gamma^0_i = \Ker(
            \widetilde \Delta^0_i \to \Delta^0_i
        )
    \), then
    \(
        N = \langle
            X_{ij}(I_{ij}),
            X_i(\Gamma^0_i)
        \rangle \leq \stunit(R, \Delta)
    \) is a central subgroup. It is easy to see using \cite[lemma 11]{thesis} that
    \[
        N \cap \bigl(
            \prod_{0 \neq |i| < j}
                X_{ij}(\widetilde R_{ij})
            \times \prod_{i > 0}
                X_i(\widetilde \Delta^0_i)
        \bigr)
        = \prod_{0 \neq |i| < j} X_{ij}(I_{ij})
        \times \prod_{i > 0} X_i(\Gamma^0_i).
    \]
    It follows that the homomorphisms
    \(
        X_{ij}
        \colon R_{ij}
        \to \stunit(R, \Delta) / N
    \) and
    \(
        X_i
        \colon \Delta^0_i
        \to \stunit(R, \Delta) / N
    \) are injective and their images satisfy \ref{comm-rel}--\ref{short-gen}. The second claim follows from lemmas \ref{centerless}, \ref{ucext}, and theorem \ref{coord-cryst}.
\end{proof}

\bibliographystyle{plain}
\bibliography{references}

\end{document}